\documentclass[11pt]{amsart}
\usepackage{amsmath,amsfonts,amssymb,amsthm,bm,mathtools,enumitem}
\usepackage[hidelinks]{hyperref}
\usepackage{fullpage,epigraph}
\usepackage[dvipsnames]{xcolor}
\usepackage[capitalize]{cleveref}
\usepackage[color=Purple!50!white,textsize=tiny]{todonotes}
\usepackage[numbers,sort]{natbib}
\usepackage{tikz,mathdots}
\tikzset{vert/.style={circle, fill=black, inner sep=1.5pt}}

\newtheorem{theorem}{Theorem}[section]
\newtheorem*{theorem*}{Theorem}
\newtheorem{lemma}[theorem]{Lemma}
\newtheorem{proposition}[theorem]{Proposition}
\newtheorem{corollary}[theorem]{Corollary}

\theoremstyle{definition}
\newtheorem{definition}[theorem]{Definition}

\crefname{equation}{equation}{equations}
\crefname{lem}{Lemma}{Lemmas}
\crefname{thm}{Theorem}{Theorems}

\newlist{lemenum}{enumerate}{1}
\setlist[lemenum]{label=(\alph*), ref=\thelemma(\alph*)}
\crefalias{lemenumi}{lemma} 

\AddToHook{env/lemma/begin}{\crefalias{theorem}{lemma}}
\AddToHook{env/proposition/begin}{\crefalias{theorem}{proposition}}
\AddToHook{env/corollary/begin}{\crefalias{theorem}{corollary}}
\AddToHook{env/conjecture/begin}{\crefalias{theorem}{conjecture}}
\AddToHook{env/claim/begin}{\crefalias{theorem}{claim}}
\AddToHook{env/question/begin}{\crefalias{theorem}{question}} 
\AddToHook{env/definition/begin}{\crefalias{theorem}{definition}} 
\AddToHook{env/remark/begin}{\crefalias{theorem}{remark}} 
\AddToHook{env/example/begin}{\crefalias{theorem}{example}} 

\DeclareMathOperator\pr{Pr}
\DeclareMathOperator\ext{ext}

\newcommand\up[1]{^{(#1)}}
\newcommand\wh[1]{\widehat{#1}}
\newcommand\wt[1]{\widetilde{#1}}
\newcommand\ab[1]{\lvert#1\rvert}

\newcommand{\flo}[1]{\lfloor #1 \rfloor}
\newcommand{\bflo}[1]{\left\lfloor #1 \right\rfloor}

\newcommand\E{\mathbb{E}}

\newcommand{\Z}{\mathbf{Z}}
\newcommand{\Q}{\mathbf{Q}}
\newcommand{\A}{\mathcal{A}}
\newcommand{\G}{\mathcal{G}}
\newcommand{\K}{\mathcal{K}}

\newcommand\en{\mathcal{E}}
\newcommand\largereal{R}

\let\leq\leqslant
\let\geq\geqslant

\title{Finding blowups one vertex at a time}
\author{Jacob Fox}
\author{Yuval Wigderson}
\author{Yunkun Zhou}
\date{}
\thanks{JF: Department of Mathematics, Stanford University, Stanford, CA 94305. Email: \texttt{jacobfox@stanford.edu}. Research
supported by NSF awards DMS-2154129 and DMS-2452737.\\
\indent YW: Institute for Theoretical Studies, ETH Z\"urich, 8006 Z\"urich, Switzerland. Email: \texttt{yuval.wigderson@eth-its. ethz.ch}. Research supported by Dr.\ Max R\"ossler, the Walter Haefner Foundation, and the ETH Z\"urich Foundation.\\
\indent YZ: Email: \texttt{yunkunzhou@alumni.stanford.edu}.}

\begin{document}
\begin{abstract}
	An influential theorem of Nikiforov states that if an $N$-vertex graph $G$ contains at least $\gamma N^h$ copies of some fixed $h$-vertex graph $H$, then $G$ contains an $H$-blowup of order $c_H(\gamma)\log N$. We provide a new proof of this theorem, which in particular improves the best known bound on the constant $c_H(\gamma)$. In contrast to previous proofs, our proof is iterative, finding the blowup one vertex at a time. 
\end{abstract}

\maketitle

\section{Introduction}
Given an integer $t$ and a graph $H$, its \emph{$t$-blowup} $H[t]$ is the graph obtained from $H$ by replacing every vertex by an independent set of order $t$, and every edge by a complete bipartite graph between the two corresponding independent sets. Blowups occupy a central role in extremal graph theory, and many of the field's central theorems can be phrased as statements about blowups, including the Erd\H os--Stone theorem \cite{MR0018807} and the K\H ov\'ari--S\'os--Tur\'an theorem \cite{MR0065617}, which states that if an $N$-vertex graph contains at least $\gamma N^2$ copies of $K_2$ (i.e.\ edges), then it contains a copy of $K_2[t]$, for $t = \Omega_\gamma(\log N)$. The following remarkable and influential theorem of Nikiforov \cite{MR2409174,MR2398823} extends this to yield a large blowup of any graph.
\begin{theorem}\label{thm:nikiforov}
	Let $H$ be an $h$-vertex graph and let $\gamma>0$. If $N$ is sufficiently large, then every $N$-vertex graph with at least $\gamma N^h$ copies of $H$ contains an $H$-blowup $H[t]$, where
	\[
		t \geq c_H(\gamma)\log N,
	\]
	and $c_H(\gamma)>0$ is a constant depending only on $\gamma$ and $H$. 
\end{theorem}
By considering a random graph with edge density $\gamma^{1/e(H)}$, it is not hard to show that this result is best possible up to the value of $c_H(\gamma)$ for nonempty $H$.  Nikiforov's theorem has several important consequences, among them a short proof of the Bollob\'as--Erd\H os--Simonovits theorem \cite{MR335347,MR485528}. 

Nikiforov's original proof of \cref{thm:nikiforov} is very short and elegant, and the technique he introduced to prove it has since found many other applications (e.g.\ \cite{MR2959395,MR4149162}). Nevertheless, \cref{thm:nikiforov} occupies a bit of an unusual space in the theory, because several natural variants of it, though easy to conjecture, appear extremely difficult to prove. Most notably, the extension of \cref{thm:nikiforov} to hypergraphs remains an outstanding open problem, whose resolution would have a number of interesting applications. 

Because of this, there has been a great deal of interest in providing alternative proofs of \cref{thm:nikiforov}, in the hopes that new perspectives would shed greater light on the problem, and hopefully lead to techniques that are more widely applicable. In particular, there is a lot of interest in understanding the quantitative aspects of \cref{thm:nikiforov}, and specifically in determining the behavior $c_H(\gamma)$ as a function of $\gamma$ (this question was perhaps first explicitly raised in \cite[Problem 2]{MR2718688}). We henceforth use the notation $c_H(\gamma)$ to denote the optimal constant in \cref{thm:nikiforov}, that is, the supremum over all $c>0$ such that every $N$-vertex graph with at least $\gamma N^h$ copies of $H$ contains an $H$-blowup of size $(c-o(1))\log N$, where the $o(1) \to 0$ as $N \to \infty$.  The only known upper bound on $c_H(\gamma)$, from an elementary analysis of the random graph $G(n,\gamma^{1/e(H)})$, shows that $c_H(\gamma) = O_H(1/{\log \frac 1 \gamma})$. By contrast, Nikiforov's proof yields a polynomial lower bound, namely $c_H(\gamma)= \Omega_H(\gamma^h)$ for every\footnote{Nikiforov actually only claimed this bound \cite{MR2409174} for $H=K_h$, and the weaker bound $c_H(\gamma)=\Omega_H(\gamma^{h^2})$ for a general $h$-vertex graph \cite{MR2398823}. However, it is not hard to see that the result for cliques immediately implies the same bound for all graphs (see \cref{prop:monotonicity reduction} below).} $h$-vertex graph $H$. This is a rather wide gap, and reflects our poor understanding of the problem. In particular, if the logarithmic upper bound does not give the true value of $c_H(\gamma)$, meaning that there is a better construction than a random graph, this may explain the difficulty of extending \cref{thm:nikiforov} to other settings, as the heuristics driving the various conjectural extensions of Nikiforov's theorem often rely on the intuition coming from random graphs and hypergraphs. 

Motivated by the twin goals of finding new proofs of \cref{thm:nikiforov}, and of improving the bounds on $c_H(\gamma)$, many researchers have studied this problem. The first improvement was due to R\"odl and Schacht \cite{MR2993136}, who gave a new proof of \cref{thm:nikiforov}, and in so doing proved that\footnote{Here and in what follows, a $o(1)$ in the exponent represents a quantity tending to $0$ as $\gamma\to 0$.} $c_H(\gamma)\geq \gamma^{1+o(1)}$ for every graph $H$. While still polynomial, their bound is much better than Nikiforov's bound of $\gamma^h$, and in particular the exponent is independent of $H$. This bound was further slightly improved by Fox, Luo, and Wigderson \cite{MR4195582}, to $c_H(\gamma)\geq \gamma^{1-1/e(H)+o(1)}$, yielding a small power savings depending on $H$. Finally, Gir\~ao, Hunter, and Wigderson \cite{MR5007509} recently proved that $c_H(\gamma)=\Theta_H(1/{\log \frac 1 \gamma})$ whenever $H$ is triangle-free, thus completely resolving the problem in this case. 

Our main result in the present paper is a further polynomial improvement for all graphs.
\begin{theorem}\label{thm:main}
	For every graph $H$, we have $c_H(\gamma)\geq \gamma^{1/2 +o(1)}$. 

	More precisely, if $H$ has $h$ vertices and $N$ is sufficiently large with respect to $\gamma$, then every $N$-vertex graph with at least $\gamma N^h$ copies of $H$ contains an $H$-blowup $H[t]$, where
	\[
		t \geq \frac{\delta_H\sqrt \gamma }{(\log \frac 1 \gamma)^2}\log N
	\]
	and $\delta_H>0$ is a constant depending only on $H$. 
\end{theorem}
In particular, our result is the first to obtain a power savings over the R\"odl--Schacht bound that is independent of $H$. 

However, we believe that our main contribution is not the precise bound $c_H(\gamma)\geq \gamma^{1/2+o(1)}$, but rather the new technique we introduce, and we are hopeful that this technique will prove useful in related settings. We provide a detailed overview of our technique in \cref{sec:sketch}, but let us briefly remark on the differences between our approach and previous ones. Like all of the previous proofs \cite{MR2409174,MR2398823,MR2993136,MR4195582,MR5007509}, we find the parts of the blowup $H[t]$ one at a time. However, in these earlier proofs, each blowup part is found in a ``one-shot'' way, by applying (a variant of) the K\H ov\'ari--S\'os--Tur\'an theorem to some appropriate auxiliary bipartite graph. In our approach, by contrast, we find the $t$ vertices of a part of $H[t]$ one vertex at a time, picking them out via a greedy algorithm. In order to analyze the evolution of this algorithm, we track an associated ``energy'' functional. This iterative approach bears some resemblance to the recent \emph{book algorithm} used in the spectacular improvement to the upper bound on diagonal Ramsey numbers \cite{MR5067166} (see also \cite{2407.19026,MR5057671} for variants of this technique, as well as the surveys \cite{bourbaki,2601.05221}), as well as to the technique used to prove Szemer\'edi's regularity lemma \cite{MR540024} and its many extensions. Because such iterative approaches have found applications in a wide variety of areas, including in hypergraphs, we are cautiously optimistic that our technique could also be adapted to these settings. These connections suggest many potential avenues for further research, and we speculate on some of these in \cref{sec:conclusion}.

We conclude with a few further remarks on the quantitative aspects of \cref{thm:main}. First, by a more careful analysis of our proof, one of the logarithmic factors in the denominator can be removed; however, since there is no reason to expect the resulting bound to be optimal, we have chosen to present a slightly weaker result with a more transparent proof. Second, while we do not expect the bound $\gamma^{1/2+o(1)}$ to be optimal, it appears to be a hard limit of our technique; concretely, if $G$ is a random $N$-vertex tripartite graph with bipartite edge densities $\frac 12, \sqrt{\gamma},\sqrt{\gamma}$, respectively, then $G$ contains $\Omega(\gamma N^3)$ triangles as well as a copy of $K_3[\Omega(\log N/{\log \frac 1 \gamma})]$, and yet we do not know how to use our technique to find a copy of $K_3[t]$ in $G$ for any $t = \omega(\sqrt \gamma \log N)$.

The rest of this paper is organized as follows. As our proof is fairly technical, we give a detailed overview of it in \cref{sec:sketch}. \cref{sec:prelims} contains some preliminary lemmas we use in our proof, especially the various technical computations related to our energy function $\en_\largereal$. \cref{sec:basic} rigorously defines and analyzes the basic iterative algorithm discussed above, and \cref{sec:switching} does the same for the more refined algorithm involving switching. Using these tools, we present the proof of \cref{thm:main} in \cref{sec:main proof}. We conclude in \cref{sec:conclusion} with some concluding remarks. 

\section{Proof overview}\label{sec:sketch}
We now discuss our proof of \cref{thm:main} in more detail. First, while \cref{thm:main} is stated for all graphs $H$, all of the difficulty actually lies in proving it in the case $H=K_3$. Indeed, there is a simple and well-known reduction (\cref{prop:monotonicity reduction}) that shows that it suffices to prove \cref{thm:main} for cliques; moreover, our main argument also allows us to reduce the problem for any clique to the case of the triangle. So for the moment we focus on $H=K_3$, and at the end of the proof sketch briefly discuss how the same ideas yield the reduction from all cliques to this case. 

Let now $G$ be a graph with $\gamma N^3$ triangles. Our next reduction allows us to find within $G$ three sets $A,B,C$, each of size $n = \Omega_\gamma(N)$, with the following properties: each edge in $E(A,B)$ lies in at least $\tau\ab C$ triangles of $G$, for some $\tau \geq \gamma^{2/3}$, and moreover the number of edges between $A$ and $B$ is $p\ab A \ab B$, for some $p \geq \gamma/\tau$. We find this structure by applying the regularity lemma\footnote{More precisely, we apply the cylinder regularity lemma of Duke, Lefmann, and R\"odl \cite{MR1333857}, which gives us a much more reasonable ``sufficiently large'' condition on $N$.} to $G$, and passing to a regular triple of parts $A,B,C$ spanning at least $\gamma \ab A \ab B \ab C$ triangles, which must exist by averaging. 
By the counting lemma, the number of triangles among these three parts is very close to $d(A,B)d(A,C)d(B,C)\ab A \ab B \ab C$, where $d$ denotes the edge density between two parts. Without loss of generality $p=d(A,B)$ is the smallest of the three edge densities, implying that $\tau=d(A,C)d(B,C)\geq \gamma^{2/3}$. We've now found that an \emph{average} edge in $E(A,B)$ lies on at least $\tau\ab C$ triangles, and that $p\tau=d(A,B)d(A,C)d(B,C) \geq \gamma$, so we are essentially done; the final step is to again use properties of regular pairs to argue that nearly all the edges have the average behavior, so upon deleting a few outliers we obtain the claimed structure.

We now turn to our main argument. Our goal is to pick out vertices $c_1,\dots,c_s \in C$, and to restrict to their common neighborhoods $A^\dagger \subseteq A, B^\dagger \subseteq B$. If we can then find a complete bipartite subgraph $K_{t,t}$ between $A^\dagger$ and $B^\dagger$, we have found a complete tripartite graph $K_{s,t,t}$ together with $C^\dagger=\{c_1,\dots,c_s\}$; it is for this reason that we insist that all vertices in $A^\dagger\cup B^\dagger$ are adjacent to all of $c_1,\dots,c_s$. In particular, this yields a triangle blowup $K_3[\min\{s,t\}]$. 

In order to find a copy of $K_{t,t}$ in $A^\dagger \cup B^\dagger$, we will apply the K\H ov\'ari--S\'os--Tur\'an theorem \cite{MR0065617}, which states that every bipartite graph $(X,Y)$ contains a complete bipartite subgraph with parts of size $\Omega(\log \min \{\ab X,\ab Y\}/{\log(1/d(X,Y))})$; this essentially says that for $H=K_2$, we have the optimal dependence $c_{K_2}(\gamma)\geq \Omega(1/{\log \frac 1 \gamma})$ in \cref{thm:nikiforov}. 

Note that the K\H ov\'ari--S\'os--Tur\'an theorem has a logarithmic dependence on both the sizes of the sets and on the edge density, which means that we are allowed to lose a great deal when passing from the pair $(A,B)$ to the pair $(A^\dagger,B^\dagger)$. Concretely, it is sufficient for us if $\ab{A^\dagger},\ab{B^\dagger}\geq \sqrt n$, and if $d(A^\dagger,B^\dagger)\geq p^2$, where we recall that $\ab A = \ab B = n$ and $d(A,B)=p$. That is, we are allowed to lose by a polynomial amount on both the sizes of the sets and on the edge density. The reason is that
\[
	\frac{\log \sqrt n}{\log \frac{1}{p^2}} = \frac 14 \cdot \frac{\log n}{\log \frac 1 p},
\]
hence any polynomial loss yields an absolute constant factor loss in the final bound.

In order to ensure we obtain these bounds on $\ab{A^\dagger},\ab{B^\dagger}$, and $d(A^\dagger,B^\dagger)$, it will be helpful to track a certain potential function, whose inputs are the three numbers $\ab A, \ab B, d(A,B)$. Whatever potential function we pick, we would like the effect of dividing $\ab A$ or $\ab B$ by $\sqrt n$ to be roughly the same as the effect of multiplying $d(A,B)$ by $p$, since we are allowing ourselves to lose a factor of $\sqrt n$ on the former, but only a factor of $p$ on the latter. This suggests that a natural form of the potential function multiplies $\ab A \ab B$ by a high power of $d(A,B)$, precisely so that losses in $d(A,B)$ will be amplified. Concretely, for a real number $\largereal>0$ we define the \emph{$\largereal$-energy} of the pair $(A,B)$ to be
\[
	\en_\largereal(A,B) = \ab A \ab B d(A,B)^\largereal.
\]
Thus, dividing $\ab A$ or $\ab B$ by $\sqrt n$ decreases $\en_\largereal(A,B)$ by a factor of $\sqrt n$, whereas multiplying $d(A,B)$ by a factor of $p$ decreases $\en_\largereal(A,B)$ by a factor of $p^R$. As we want these effects to be roughly the same,
in our proof, we will end up selecting $\largereal = \Theta(\log n/{\log \frac 1p})$. 

As discussed in the introduction, we will pick out the vertices $c_1,\dots,c_s$ one at a time. So suppose that at some point in the process, we have picked out $c_1,\dots,c_\ell$, and restricted $A$ and $B$ to their common neighborhoods, say $A\up \ell, B\up \ell$. Inductively, we have maintained that $A\up \ell, B\up \ell$, and $d(A\up \ell, B\up \ell)$ are fairly large, namely at most polynomially smaller than their starting values. 
Our goal is now to select $c_{\ell+1}$ in such a way that we can keep the iteration going.

To do so, let us consider some $c^* \in C$, and let $A^+=N(c^*)\cap A\up \ell, B^+=N(c^*)\cap B\up \ell$. Our hope is to show that, by selecting $c^*$ appropriately, we can ensure that $\ab{A^+},\ab{B^+}$, and $d(A^+,B^+)$ remain fairly large, so that we can set $c_{\ell+1}=c^*, A\up{\ell+1}=A^+, B\up{\ell+1}=B^+$, and hence continue the iteration. 

First, we note that controlling the size of $A^+$ and $B^+$ is not difficult. Indeed, since every edge in $E(A,B)$ lies in at least $\tau \ab C$ triangles, we see that the triple $(A\up \ell, B\up \ell, C)$ still contains at least $\tau e(A\up \ell, B\up \ell)\ab C$ triangles, and hence we can select\footnote{In reality, we also need to ensure that $c^*$ is a new vertex, i.e.\ we must pick $c^* \in C \setminus \{c_1,\dots,c_\ell\}$. But as $\ell \leq s = O_\gamma(\log n) \ll n$, this restriction has almost no impact, and we ignore it for this overview.} $c^* \in C$ which lies in at least $\tau e(A\up \ell, B\up \ell)$ triangles. But the number of triangles containing $c^*$ is precisely $e(A^+,B^+)$, hence
\[
	\ab{A^+}\ab{B^+}\geq e(A^+,B^+) \geq \tau e(A\up \ell, B\up \ell) = \tau p\up \ell \ab {A\up\ell} \ab {B\up \ell},
\]
where we set $p\up \ell = d(A\up \ell, B\up \ell)$. Recall that at the beginning of the process, we had $\tau p \geq \gamma$, and by our inductive assumption we have $p\up \ell \geq p^2$, hence $\ab{A^+}\ab{B^+}\geq \gamma^2 \ab{A\up \ell}\ab{B\up \ell}$. Note that losing a factor of $\gamma^2$ on the sizes of the sets at every step is completely fine for us: we will end up doing $s$ steps, so the amount we will lose on the sizes is $\gamma^{2s}$. So as long as $s \leq \log n/(10\log \frac 1 \gamma)$, say, the total amount that we will lose on the sizes of $A$ and $B$ is at most a small power of $n$, which is exactly what we are targeting. In particular, if this were the only constraint on our process, we could keep it going until $s = \Theta(\log n/{\log \frac 1 \gamma})$, thus obtaining the optimal bound $c_{K_3}(\gamma) = \Theta(1/{\log \frac 1 \gamma})$. Unsurprisingly, therefore, this is not the only, or even the main, constraint on our process.

Indeed, as discussed above, and as recorded in the form of our energy function $\en_\largereal$, we are willing to tolerate much larger losses on the sizes of $\ab{A^+},\ab{B^+}$ than we are on the density $d(A^+,B^+)$. Specifically, let us pick some $\alpha>0$, and say that we are happy if we can ensure that $d(A^+,B^+)\geq p\up \ell - \alpha$. Since we will eventually do $s$ such steps, we will choose $\alpha \leq p/(2s)$, so that at the end of the day we will still have $d(A^\dagger,B^\dagger) \geq p - \alpha s \geq p/2$, which is even stronger than the lower bound of $p^2$ we were originally targeting\footnote{As mentioned in the introduction, we can remove one of the logarithmic factors in \cref{thm:main} via a more careful analysis, and the slack in this step is the reason: to obtain this stronger result, we would pick a slightly larger value of $\alpha$, so that $p - \alpha s = p^2$.}. 

As we are happy if $d(A^+,B^+)\geq p\up \ell - \alpha$, let us assume contrariwise that $d(A^+,B^+)<p\up \ell-\alpha$. In this case, we use a standard but crucial trick, which goes back at least to \cite{MR928742}: we have found a large pair of subsets $A^+\subseteq A, B^+\subseteq B$ which span substantially fewer edges than what you would expect based on the global average, and hence these ``missing edges'' must be somewhere else. More precisely, setting $A^-=A\up \ell \setminus A^+, B^- = B\up \ell \setminus B^+$, we learn that one of the three pairs $(A^-,B^-),(A^+,B^-),(A^-,B^+)$ must have substantially \emph{higher} edge density than $p\up \ell$. Specifically, it is not hard to see that for some pair $(X,Y) \in \{(A^-,B^-),(A^+,B^-),(A^-,B^+)\}$, we have
\begin{equation}\label{eq:discrepancy consequence}
	d(X,Y) \geq p\up \ell + \frac{\alpha\ab{A^+}\ab{B^+}}{3\ab X \ab Y}
\end{equation}
Because our energy functional $\en_\largereal$ is highly sensitive to changes in the edge density, an elementary but tedious computation in fact implies that for one of these three pairs $(X,Y)$, we actually have $\en_\largereal(X,Y)>\en_\largereal(A\up \ell,B\up \ell)$, so long as $R \geq 10 p/(\alpha \tau)$. That is, in the situation where we are not happy because $d(A^+,B^+)< p\up \ell-\alpha$, we have the ability to restrict $A\up \ell, B\up \ell$ to subsets and to increase the energy. Of course, such steps are not directly useful to us because we are not able to select a new vertex $c_{\ell+1}$, but in return we obtain an energy increment.

Thus, we end up with two kinds of steps: steps where we select a new vertex $c_{\ell+1}$ and restrict to its neighborhoods, at the expense of losing at most $\alpha$ on the density, and steps where we do not select a new vertex but can restrict to subsets and obtain an energy increment. In the formal analysis of the algorithm, it will be more convenient to never actually perform the second kind of step, and instead, whenever we create a new pair $(A\up \ell, B\up \ell)$, we immediately restrict these two sets to subsets which maximize the energy $\en_\largereal$. By doing this, we automatically know that no energy increment is possible, hence the second scenario above can never arise, and the vertex $c^*$ we pick necessarily satisfies $d(A^+,B^+) \geq p\up \ell-\alpha$. 

It only remains to see for how many steps $s$ we can keep this process going. We also have a free choice to make, namely the choice of the parameter $\alpha$, subject to the constraints discussed above, namely $\alpha \leq p/(2s)$ and $R \geq 10 p /(\alpha \tau)$, which imply $s \leq p/(2\alpha) \leq R\tau/20$. Recall too that we require $R = \Theta(\log n/{\log \frac 1 p})$, in order to be able to conclude at the end that $\ab{A^\dagger},\ab{B^\dagger},d(A^\dagger,B^\dagger)$ are still polynomially large relative to their starting values, so that we can apply the K\H ov\'ari--S\'os--Tur\'an theorem. This implies that we can keep the process going until step
\[
	s = \Theta(R\tau) = \Theta \left( \frac{\tau \log n}{\log \frac 1 p} \right) = \Omega \left( \frac{\gamma^{2/3}\log n}{\log \frac 1 \gamma}  \right),
\]
where we recall that $\tau \geq \gamma^{2/3}$ and $p \geq \gamma/\tau \geq \gamma$.
Moreover, at the end of the process, we apply the K\H ov\'ari--S\'os--Tur\'an theorem to $(A^\dagger, B^\dagger)$, finally using our guarantee that $\ab{A^\dagger}\geq \sqrt n, \ab{B^\dagger}\geq \sqrt n, d(A^\dagger,B^\dagger)\geq p^2$, to find a copy of $K_{t,t}$ for
\[
	t = \Omega \left( \frac{\log \sqrt n}{\log \frac{1}{p^2}} \right) = \Omega \left( \frac{\log n}{\log \frac 1 \gamma} \right).
\]
In particular, $t$ is much larger than $s$, and we conclude that $G$ contains a triangle blowup $K_3[s]$, 
proving that $c_{K_3}(\gamma)\geq \gamma^{2/3+o(1)}$. 

Of course, our goal is to actually prove that $c_{K_3}(\gamma)\geq \gamma^{1/2+o(1)}$, and we now turn to describe the remaining ingredient necessary to achieve this. At a high level, the new ingredient is to exploit the symmetry between the three parts: the argument above treated $C$ completely differently from $A,B$, and there are situations where this becomes wasteful.

More precisely, in the argument above, 
the weakness is that $t$ is much larger than $s$, by a factor of roughly $1/\tau$. Thus, if we could somehow pass to a configuration with a larger value of $\tau$, we would obtain a larger value of $s$, and hence a better lower bound on $c_{K_3}(\gamma)$. 

How could we achieve this? Note that at every step of the process, the total number of triangles present is at least $e(A\up \ell, B\up \ell) \cdot (\tau \ab C) \geq (p\tau/2)\ab{A\up \ell}\ab{B\up \ell}\ab C\geq (\gamma/2)\ab{A\up \ell}\ab{B\up \ell}\ab C$, recalling that we have maintained $p\up \ell = d(A\up \ell, B\up \ell) \geq p/2$. Consequently, the average number of triangles that an edge of $E(A\up \ell, C)$ lies in is at least
\[
	\frac{(\gamma/2)\ab{A\up \ell}\ab{B\up \ell}\ab C}{e(A\up \ell,C)} \geq  \frac{\gamma/2}{d(A\up \ell,C)} \ab{B\up \ell}.
\]
In particular, if it ever happens that $\wh \tau \coloneqq \gamma/d(A\up \ell,C) \gg \tau$, then we have essentially found the configuration of the type discussed above, where the value of $\tau$ has been replaced by the much larger value $\wh \tau/2$. More precisely, letting $\wh A = A\up \ell, \wh B = C, \wh C = B\up \ell$, we see that the triple $(\wh A, \wh B, \wh C)$ contains at least $(\gamma/2)\ab{\wh A}\ab{\wh B}\ab{\wh C}$ triangles, and that an average edge in $E(\wh A, \wh B)$ lies in at least $(\wh \tau/2) \ab{\wh C}$ triangles. This is not, per se, the setup we had above, since we only have a bound on the \emph{average} number of triangles containing an edge in $E(\wh A, \wh B)$, but by deleting the edges lying in too few triangles, we can correct this issue. By doing this and then running the argument above, we end up finding a triangle blowup $K_3[\wh s]$, where
\[
	\wh s = \Omega \left( \frac{\wh \tau \log n}{\log \frac 1 \gamma} \right). 
\]

Let us pick some parameter $\mu \gg \tau$, which we will optimize later, and declare that we will do such a switch if ever $\wh \tau > \mu$. In particular, if this ever happens, then we find a triangle blowup of size $\wh s = \Omega(\mu \log n/{\log \frac 1 \gamma})$. 

We may thus assume that we never perform such a switch, and hence that at every stage of the process we have $\wh \tau \leq \mu$, or equivalently $d(A\up \ell,C)\geq \gamma/\mu$. By the same argument but with the roles of $A\up \ell$ and $B\up \ell$ interchanged, we may also assume that $d(B\up \ell,C)\geq \gamma/\mu$. This implies that the average degree of a vertex $c^* \in C$ in $A\up \ell$ is at least $(\gamma/\mu)\ab{A\up \ell}$, and similarly for the degree into $B\up \ell$. By another round of removing vertices whose degree is much lower than the average, this means that in the argument presented above, we may assume that $\ab{A^+}\geq (\gamma/\mu)\ab{A\up \ell}, \ab{B^+}\geq (\gamma/\mu)\ab{B\up \ell}$. 

The reason this is useful to us is that in \eqref{eq:discrepancy consequence}, the amount we gain on the density of the pair $(X,Y)$ is proportional to $\ab{A^+}\ab{B^+}$, and hence we are gaining more here. This, in turn, means that we can find $\en_\largereal(X,Y)>\en_\largereal(A\up \ell,B\up \ell)$ under a more modest assumption on $\largereal$; it now suffices that $\largereal \gtrsim \mu/(\alpha \tau)$,
where the notation $\gtrsim$ hides logarithmic factors. Recall that earlier we needed $R \geq 10p/(\alpha \tau)$, and hence we now obtain a milder requirement, so long as $\mu \ll p$. 

It now remains to do the optimization; we again require $\alpha \leq p/(2s)$, so we conclude that we can continue the process up to step
\[
	s \leq \frac{p}{2\alpha} \lesssim \frac{p\tau R}{\mu} = \wt\Theta \left( \frac{\gamma}{\mu} \log n \right),
\]
where we hide factors of $\log \frac 1 \gamma$ in the $\wt \Theta$ notation, and recall that we choose $R = \wt \Theta(\log n)$.
In order to have the strongest possible bound, we should choose $\mu$ so that $s$ and $\wh s$ are of roughly the same order, meaning that we should take $\gamma/\mu=\mu$, or equivalently $\mu=\sqrt \gamma$. By doing so, we conclude that $c_{K_3}(\gamma) \geq \gamma^{1/2+o(1)}$. 

Note that this argument is useful so long as $\tau \ll \mu \ll p$; we also know that $p \tau = \Omega(\gamma)$. So we can only win if $\tau \ll p$, or equivalently if $p \gg \sqrt \gamma$. In the regime where $p \asymp \tau \asymp \sqrt \gamma$, this switching argument wins us nothing. This is precisely why, in the introduction, we said that our technique does not seem able to cope with a random tripartite graph of edge densities $\frac 12, \sqrt \gamma, \sqrt \gamma$; in such a graph, no matter which of the three parts we select as $C$, we will have $p \geq \sqrt \gamma$.

To conclude, we briefly explain how we obtain the same result for $c_{K_{k+1}}(\gamma)$, using essentially the same argument. So let $G$ be an $N$-vertex graph with $\gamma N^{k+1}$ copies of $K_{k+1}$.  First, by applying the same regularity technique, we can pass to a collection of disjoint sets $A_1,\dots,A_k,C \subseteq V(G)$, each of size $n=\Omega_{\gamma,k}(N)$, with the property that there are at least $p\ab{A_1}\dotsb \ab{A_k}$ copies of $K_k$ among $(A_1,\dots,A_k)$, each of which extends to at least $\tau \ab C$ copies of $K_{k+1}$ with a vertex in $C$, where $\tau \geq \gamma^{2/(k+1)}$ and $p\tau \geq \gamma$. We now run exactly the same iterative argument, selecting vertices from $C$ one by one and restricting to their common neighborhoods in $A_1,\dots,A_k$. The analysis of this algorithm is nearly identical to the one above, and the upshot of it is that we can pass to some $C^\dagger \subseteq C, A_i^\dagger \subseteq A_i$, with the following properties: we have that $\ab{A_i^\dagger}= \sqrt n$ for all $i$, that $\ab {C^\dagger}=s = \Theta(\tau \log n/{\log \frac 1 p})$, and that the density of $K_k$-copies among $A_1^\dagger,\dots,A_k^\dagger$ is at least $p/2$. We now note that the induced subgraph on $A_1^\dagger \cup \dots \cup A_k^\dagger$ has $N'=k\sqrt n = \Omega(\sqrt{N})$ vertices and at least $\Omega(\gamma (N')^k)$ copies of $K_k$, so we may proceed by induction on $k$, with the base case $k=2$ being precisely the triangle case handled above. The key thing to note is that for $k \geq 3$, we have $k+1 \geq 4$, and hence
\[
	s = \Omega \left( \frac{\gamma^{2/(k+1)}\log n}{\log \frac 1 \gamma} \right) \geq \Omega \left( \frac{\gamma^{1/2}\log n}{\log \frac 1 \gamma} \right),
\]
meaning that the number of vertices we select from $C$ is indeed enough.
In fact, the same argument yields the recursive bound
\[
	c_{K_{k+1}}(\gamma) \geq \min \{\gamma^{2/(k+1)+o(1)}, c_{K_k}(\Omega_k(\gamma))\}.
\]

\section{Preliminaries}\label{sec:prelims}
We will use the K\H ov\'ari--S\'os--Tur\'an theorem \cite{MR0065617} in the following form. For a bipartite graph with parts $(A,B)$, we denote by $d(A,B) = e(A,B)/(\ab A \ab B)$ its edge density. 
\begin{theorem}\label{thm:kst}
	Let $p \in (0,\frac 12]$, let $G$ be a bipartite graph with parts $(X,Y)$, and assume that $d(X,Y)\geq p$. Then $G$ contains a copy of $K_{t,t}$, where
	\[
		t = \bflo{\frac{\log (\min\{\ab X, \ab Y\})}{2\log \frac 1p}}.
	\]
\end{theorem}
\subsection{Analytic inequalities}
In this paper, all logarithms are to base $2$. 

We need the following lemma, a simple analytic inequality.
\begin{lemma}\label{lem:bernoulli}
	For all real numbers $y\geq 0 ,\largereal \geq 1$, we have
	\[
		y^{1/\largereal}\leq 1+\frac y \largereal . 
	\]
\end{lemma}
\begin{proof}
	Recall Bernoulli's inequality, which states that $(1+x)^\largereal\geq 1+\largereal x$ for all $\largereal \geq 1$ and $x \geq 0$. Setting $x=y/\largereal$, we find that
	\[
		\left( 1+\frac y\largereal \right)^\largereal = (1+x)^\largereal \geq 1+\largereal x = 1+y \geq y,
	\]
	which yields the claimed bound by taking $1/\largereal$ powers of both sides. 
\end{proof}
The following is a variant of \cref{lem:bernoulli}, giving us a better upper bound if we know that $y$ is either very small or very large. 
\begin{lemma}\label{lem:stronger bernoulli}
	Let $\eta \in (0,\frac 1{4}]$ and let $\largereal\geq \log \frac 1 \eta$. For every $y \in [1,(1-\eta)^{-2}] \cup [\frac 1\eta,\infty)$, we have
	\[
		y^{1/\largereal}\leq 1+(\eta \log \tfrac 1 \eta)\frac y \largereal.
	\]
\end{lemma}
\begin{proof}
	First suppose that $y \in [1,(1-\eta)^{-2}]$. By applying Bernoulli's inequality to $x=y(\eta \log \frac 1 \eta)/\largereal$, we find
	\[
		\left( 1+(\eta \log \tfrac 1 \eta)\frac y \largereal \right)^\largereal \geq 1+(\eta \log \tfrac 1 \eta)y \geq y \left( (1-\eta)^{2}+\eta\log \tfrac 1 \eta \right),
	\]
	where the second inequality uses our assumption $y \leq (1-\eta)^{-2}$. Now note that
	\[
		(1-\eta)^2+\eta \log \tfrac 1 \eta = 1+\eta(-2+\eta+\log \tfrac 1 \eta) \geq 1 + \eta(-2+2)= 1,
	\]
	using our assumptions that $\eta > 0$ and $\eta \leq \frac 14$, the latter of which implies $\log \frac 1 \eta \geq 2$. This completes the proof in the first case, upon taking the $1/\largereal$ power of both sides.

	Now suppose that $y \geq 1/\eta$. Let $\theta = (\log \frac 1 \eta)/\largereal \leq 1$ and $z = \eta y\geq 1$. We now have that
	\begin{align*}
		\left( 1+(\eta \log \tfrac 1 \eta)\frac y \largereal \right)^{1/\theta}&=\left( 1+\theta z \right)^{1/\theta}\geq 1+z\geq 2 \sqrt z\geq 2z^{1/{\log \frac 1 \eta}} = \left( \frac{z}{\eta} \right)^{1/{\log \frac 1 \eta}}=y^{1/{\log \frac 1 \eta}},
	\end{align*}
	where the first inequality is Bernoulli's, the second is AM--GM, and the third uses that $z \geq 1$ and $\log \frac 1 \eta \geq 2$. Taking the $\theta$ power of both sides gives the claimed result. 
\end{proof}

\subsection{Energy and strictly balanced subgraphs}
For a $k$-partite $k$-graph $(A_1,\dots,A_k)$, we denote by $e(A_1,\dots,A_k)$ its edge count and by $d(A_1,\dots,A_k)=e(A_1,\dots,A_k)/\prod_i \ab{A_i}$ its edge density. We will frequently use the notation $\A$ as a shorthand for the tuple $(A_1,\dots,A_k)$, to minimize the number of symbols. Similarly, we will shorten tuples such as $(A_1',\dots,A_k')$ or $(A_1^*,\dots,A_k^*)$ as $\A',\A^*$, respectively. 

For a real number $\largereal>1$, we also define the \emph{$\largereal$-energy} of $\A=(A_1,\dots,A_k)$ to be
\[
	\en_\largereal(\A) \coloneqq \ab{A_1}\dotsb\ab{A_k} d(\A)^\largereal =e(\A)d(\A)^{\largereal-1},
\]
with the convention that $\en_\largereal(\A)=0$ if $\A$ has no edges (and in particular if some $A_i$ is empty).
Finally, we say that $\A$ is \emph{strictly $\largereal$-balanced} if for all subsets $A'_i \subseteq A_i$ such that at least one inclusion is proper, we have 
\[
	\en_\largereal(\A) > \en_\largereal(\A'). 
\]
We have the following extremely simple lemma.
\begin{lemma}\label{lem:pass to max}
	Every $k$-partite $k$-graph $\A$ contains a strictly $\largereal$-balanced subgraph $\A^*$ with $\en_\largereal(\A^*)\geq \en_\largereal(\A)$. 
\end{lemma}
\begin{proof}
	Let $A^*_1\subseteq A_1,\dots, A_k^*\subseteq A_k$ be subsets which maximize $\en_\largereal(\A^*)$ among all options; if there are multiple choices achieving the same value, pick one which minimizes $\sum_i \ab{A_i^*}$. As $\A$ is one valid choice, we certainly have $\en_\largereal(\A^*)\geq \en_\largereal(\A)$. Moreover, any subsets $A_i' \subseteq A_i^*$ such that at least one inclusion is proper would be another valid choice with a strictly smaller value of $\sum_i \ab{A_i'}$, hence we conclude that $\A^*$ is strictly $\largereal$-balanced. 
\end{proof}
We remark that this notion is similar to ones used in many other contexts in extremal graph theory, e.g.\ in finding a sublinear expander or in almost-regularizing a graph. In such cases, one picks an appropriate functional of a graph, which is usually some power of its edge count divided by its vertex count, and then passes to a subgraph which maximizes this quantity. 

The main property we will need about strictly $\largereal$-balanced hypergraphs is the following lemma, which gives an upper bound on the number of edges of any subgraph of a strictly balanced hypergraph. The proof is quite simple: if a subgraph had too many edges, its energy would be higher, contradicting the assumption of strict $\largereal$-balancedness.
\begin{lemma}\label{lem:edge count in balanced}
	Let $\A$ be strictly $\largereal$-balanced, let $p=d(\A)$, and let $X_i \subseteq A_i$. Let $\mathcal X =(X_1,\dots,X_k)$. 
	\begin{lemenum}
	
	\item We have\label{lemit:edge count standard}
	\[
		e(\mathcal X) \leq p\prod_i \ab{X_i} +\frac{p}{\largereal}\prod_i \ab{A_i}.
	\]
	
	\item Let $\eta \in (0,\frac 14]$ and suppose that $\largereal \geq \log \frac 1 \eta$. If, furthermore, we have that $\frac{\prod\ab{X_i}}{\prod \ab{A_i}} \in [0,\eta] \cup [(1-\eta)^2,1]$, then we have the stronger bound
	\[
		e(\mathcal X) \leq p\prod_i \ab{X_i} + (\eta \log \tfrac 1 \eta)\frac{p}{\largereal}\prod_i \ab{A_i}.
	\]\label{lemit:edge count strengthened}
	\end{lemenum}
\end{lemma}
\begin{proof}
	There is nothing to prove if $p=0$, if some $X_i$ is empty, or if $X_i=A_i$ for all $i$, hence we may assume that $p>0$ and that at least one inclusion is proper. This implies that $\en_\largereal(\mathcal X)<\en_\largereal(\A)$. Unpacking the definition, this implies that
	\begin{equation}\label{eq:unpack energy}
		\left( \frac{d(\mathcal X)}{d(\A)} \right)^\largereal < \frac{\prod_i \ab{A_i}}{\prod_i\ab{X_i}}.
	\end{equation}
	Applying \cref{lem:bernoulli}, we conclude that
	\[
		\frac{d(\mathcal X)}{p} < \left( \frac{\prod_i \ab{A_i}}{\prod_i\ab{X_i}} \right)^{1/\largereal} \leq 1+\frac{\prod_i \ab{A_i}}{\largereal\prod_i\ab{X_i}}
	\]
	and therefore
	\[
		e(\mathcal X) = d(\mathcal X)\prod_i\ab{X_i} \leq \left( p+ \frac{p\prod_i \ab{A_i}}{\largereal\prod_i\ab{X_i}} \right)\prod_i\ab{X_i} = p\prod_i\ab{X_i} + \frac p \largereal \prod_i \ab{A_i},
	\]
	proving \cref{lemit:edge count standard}.

	Now suppose we are in the setting of \cref{lemit:edge count strengthened}. There is nothing to prove if some $X_i$ is empty, so we may assume $\prod_i\ab{X_i} \neq 0$, and then define $y = (\prod_i \ab{A_i})/(\prod_i\ab{X_i})$. By assumption we have $1/y \in [0,\eta]\cup [(1-\eta)^2,1]$, implying that either $y \geq 1/\eta$ or $1 \leq y \leq (1-\eta)^{-2}$. Hence we are in the setting of \cref{lem:stronger bernoulli}; applying this lemma to \eqref{eq:unpack energy}, we conclude that
	\[
		\frac{d(\mathcal X)}{p}\leq y^{1/\largereal} \leq 1+(\eta \log \tfrac 1 \eta) \frac y \largereal = 1+\frac{\eta \log \tfrac 1 \eta}{\largereal}\cdot \frac{\prod_i \ab{A_i}}{\prod_i\ab{X_i}},
	\]
	from which the result follows as above. 
\end{proof}

\section{The basic iteration}\label{sec:basic}
We now introduce some additional terminology. A \emph{$k$-mixed graph} $\G$ is a (non-uniform) hypergraph with $k+1$ parts, labeled $A_1,\dots,A_k, C$.  Its edges consist of some $k$-partite $k$-graph on $\A=(A_1,\dots,A_k)$, as well as some $2$-uniform bipartite graph on $(C, A_1\cup \dots \cup A_k)$. 

We denote by $N_1(c),\dots,N_k(c)$ the neighborhoods of a vertex $c \in C$ in the sets $A_1,\dots,A_k$, respectively. We say that a vertex $c\in C$ \emph{extends} an edge $(a_1,\dots,a_k) \in E(\A)$ if $c$ is adjacent to all the vertices $a_1,\dots,a_k$. We denote by $\ext(\G)$ the total number of extensions in $\G$, namely
\begin{equation}\label{eq:ext def}
	\ext(\G) = \sum_{(a_1,\dots,a_k) \in E(\A)} \sum_{c \in C}\mathbf 1_{c\text{ extends }(a_1,\dots,a_k)} = \sum_{c \in C} e(N_1(c),\dots,N_k(c)).
\end{equation}
We also denote by $\tau_C(\G)$ the \emph{minimum extension density} of $\G$, defined as 
\[
	\tau_C(\G) = \min \left\{ \frac{\ab{\{c \in C : c\text{ extends }(a_1,\dots,a_k)\}}}{\ab C} : (a_1,\dots,a_k) \in E(\A) \right\},
\]
with the convention that $\tau_C(\G)=0$ if $E(\A)=\varnothing$.
That is, for each edge in $E(\A)$, we compute what fraction of vertices in $C$ extend this edge, and then $\tau_C(G)$ is the minimum of these ratios. 
Note that $\tau_C(\G) \in [0,1]$ for every $k$-mixed graph $\G$.
The following is the key lemma of our iterative procedure. 
\begin{lemma}\label{lem:one step}
	Let $k \geq 2$ be an integer. Fix parameters $\tau\in (0,1]$ and $\largereal \geq 2^{k+1}/\tau$.
	Let $\G$ be a $k$-mixed graph with parts $\A,C$, and suppose that $\tau_C(\G)\geq \tau$ and $\A$ is strictly $\largereal$-balanced. 

	Then there exist a vertex $c^* \in C$ and non-empty subsets $A_i^*\subseteq A_i$ satisfying the following properties.
	\begin{enumerate}[label=(\roman*)]
	
	\item $c^*$ is adjacent to all vertices in $A_1^* \cup \dots \cup A_k^*$. \label{it:complete}
	\item We have \label{it:energy bound}
	\[
		\en_\largereal(\A^*)\geq 2^{-2^{k+2}/\tau}\cdot \en_\largereal(\A)
	\]
	\item $\A^*$ is strictly $\largereal$-balanced.\label{it:balanced}
	\end{enumerate}
\end{lemma}
\begin{proof}
	Let us set $p = d(\A)$ and $\alpha = 2^k p/(\largereal \tau)$, and note that $\alpha\leq p/2$ since $\largereal \tau \geq 2^{k+1}$. Note that
	\[
		\ext(\G) \geq \tau_C(\G)e(\A)\ab C \geq \tau p\ab{A_1}\dotsb\ab{A_k}\ab{C},
	\]
	and hence from \eqref{eq:ext def} we have that there is some vertex $c^* \in C$ for which
	\[
		e(N_1(c^*),\dots,N_k(c^*)) \geq \tau p \prod_i \ab{A_i}.
	\]
	We fix such a $c^*$, and let
	$A_i^+ = N_i(c^*)$ and $A_i^- = A_i \setminus A_i^+$. 
	
	Our key claim is that
	\begin{equation}\label{eq:density LB}
		d(\A^+) \geq p-\alpha. 
	\end{equation}
	Suppose for contradiction that $d(\A^+)< p-\alpha\leq p$. From the above, we find that
	\[
		p\tau \prod_i \ab{A_i} \leq e(\A^+)=d(\A^+) \prod_i \ab{A_i^+} <  p\prod_i \ab{A_i^+},
	\]
	implying that 
	\begin{equation}\label{eq:A0B0 LB}
		\frac{1}{\tau}\prod_i \ab{A_i^+} > \prod_i \ab{A_i}.
	\end{equation}
	Let $\sigma \in \{+,-\}^k$. 
	We now apply \cref{lemit:edge count standard} to $(X_1,\dots,X_k) = (A_1^{\sigma_1},\dots,A_k^{\sigma_k})\eqqcolon \A^\sigma$ to find that
	\[
		e(\A^\sigma) \leq p \prod_i \ab{A_i^{\sigma_i}} + \frac p \largereal \prod_i \ab{A_i} < p \prod_i \ab{A_i^{\sigma_i}} + \frac p {\largereal\tau} \prod_i \ab{A_i^+},
	\]
	where the second inequality is from \eqref{eq:A0B0 LB}. 
	Moreover, in the case $\sigma = (+,\dots,+)$, we actually have the stronger bound
	\[
		e(\A^+) < (p-\alpha)\prod_i \ab{A_i^+} = p\prod_i \ab{A_i^+} - \alpha \prod_i \ab{A_i^+},
	\]
	using our assumption that \eqref{eq:density LB} is false. Summing these $2^k$ inequalities, and using the fact that each edge in $\A$ lies in precisely one tuple $\A^\sigma$, we find that
	\begin{align*}
		e(\A) &= \sum_{\sigma \in \{+,-\}^k} e(\A^\sigma)\\
		&< \sum_{\sigma \in \{+,-\}^k} p \prod_i \ab{A_i^{\sigma_i}} + \left( (2^k-1) \frac{p}{\largereal \tau} - \alpha \right) \prod_i \ab{A_i^+}\\
		&< p\sum_{\sigma \in \{+,-\}^k} \prod_i \ab{A_i^{\sigma_i}} =p \prod_i \ab{A_i},
	\end{align*}
	using our choice of $\alpha$ in the second inequality. This contradicts our assumption that $p = d(\A)$, and concludes the proof of \eqref{eq:density LB}.

	We now show how to use \eqref{eq:density LB} to complete the proof of the lemma. To do this, we first compute that
	\begin{align*}
		\en_\largereal(\A^+) &=e(\A^+) d(\A^+)^{\largereal-1}\\
		&\geq \left(\tau p \prod_i \ab{A_i}\right)(p-\alpha)^{\largereal-1}\\
		&=\left(p^\largereal\prod_i \ab{A_i} \right)(\tau p^{1-\largereal}(p-\alpha)^{\largereal-1})\\
		&=\en_\largereal(\A) \left( \tau \left( 1-\frac \alpha p \right)^{\largereal-1} \right).
	\end{align*}
	Recall that $\tau\leq 1$, hence
	the inequality $\tau \geq 2^{-1/\tau}$ holds. Additionally, since $\alpha/p\leq 1/2$, the inequality $1-\alpha/p\geq 2^{-2\alpha/p}$ holds. Therefore,
	\[
		\tau \left( 1-\frac \alpha p \right)^{\largereal-1} \geq \tau \left( 1-\frac \alpha p \right)^{\largereal}\geq 2^{-1/\tau} \cdot 2^{-2\largereal\alpha/p} \geq 2^{-2^{k+2}/\tau}.
	\]
	Finally, we pass to a strictly $\largereal$-balanced subgraph $(A_1^*,\dots,A_k^*)$ of $(A_1^+,\dots,A_k^+)$, as given to us by \cref{lem:pass to max}. This automatically gives us \ref{it:balanced}, as well as \ref{it:energy bound} by the computation above. As $c^*$ is adjacent to every vertex of $A_i^+ \supseteq A_i^*$, we also have \ref{it:complete}.
\end{proof}
We remark that the proof of \cref{lem:one step} would work verbatim if we only knew that the \emph{average} number of extensions of an edge in $E(\A)$ were at least $\tau \ab C$, rather than the \emph{minimum}. However, our next result is obtained by iterating \cref{lem:one step}, and controlling the minimum is much simpler than controlling the average during such an iteration. 
\begin{proposition}\label{prop:basic iteration}
	Let $k \geq 2$ be an integer, and fix $\tau_0,p \in (0,\frac 12]$. Let $n \geq p^{-10^k/\tau_0}$, and let $\G$ be a $k$-mixed graph with parts $\A,C$, where $\min \{\ab{A_1},\dots,\ab{A_k},\ab C\} \geq n$. If $\tau_C(\G)\geq \tau_0$ and $d(\A)\geq p$, then there are subsets $A_i^\dagger \subseteq A_i, C^\dagger \subseteq C$ satisfying 
	\begin{enumerate}[label=(\alph*)]
		\item $\ab{C^\dagger}= \flo{(\tau_0 \log n)/(2^{k+5}\log \frac 1p)}$,\label{it:Cdagger size}
		\item $\ab{A_i^\dagger}\geq \sqrt n$ for every $i$,  \label{it:Adagger size}
		\item $d(\A^\dagger) \geq p/2$, and \label{it:Adagger density}
		\item Every vertex of $C^\dagger$ is adjacent to every vertex of $A_1^\dagger \cup \dots \cup A_k^\dagger$.\label{it:daggers complete}
	\end{enumerate}
\end{proposition}
\begin{proof}
	Let
	\[
		s = \bflo{\frac{\tau_0 \log n}{2^{k+5}\log \frac 1p}} \qquad \text{and} \qquad \largereal=\frac{\log n}{4\log \frac 1 p}.
	\]
	We will define a nested sequence of induced subgraphs $\G \supseteq \G\up 0 \supseteq \G\up 1 \supseteq \dots \supseteq \G\up s$ on vertex sets $(A_1\up \ell,\dots,A_k\up \ell, C\up \ell)_{\ell =0}^s$, satisfying the following properties.
	\begin{enumerate}[label=(\roman*)]
		\item $\ab{C\up\ell}=\ab C - \ell$. \label{property:C size}
		\item Every vertex in $C \setminus C\up \ell$ is adjacent to all vertices in $A_1\up \ell \cup \dots \cup A_k \up \ell$. \label{property:complete}
		\item\label{property:energy} $\A\up \ell$ is strictly $\largereal$-balanced and satisfies $$\en_\largereal(\A\up\ell) \geq 2^{-2^{k+3}\ell/\tau_0} p^{\largereal}\prod_i \ab{A_i}.$$ 
	\end{enumerate}
	To begin, we set $C\up 0 = C$, so that properties \ref{property:C size} and \ref{property:complete} are trivially satisfied. We also let $\A\up 0$ be a strictly $\largereal$-balanced subgraph of $\A$, as given by \cref{lem:pass to max}, for which we have
	\[
		\en_\largereal(\A\up 0) \geq \en_\largereal(\A)=d(\A)^\largereal \prod_i\ab{A_i}\geq p^\largereal \prod_{i}\ab{A_i},
	\]
	so \ref{property:energy} is also satisfied for $\ell=0$.

	Inductively, suppose we have defined $\G\up \ell$ for some $\ell<s$, and aim to define $\G\up{\ell+1}$. 
	Note that $s \leq \tau_0 n/2$, hence $\ab{C \setminus C\up \ell} = \ell < \tau_0 n/2 \leq \tau_0 \ab C/2$ by \ref{property:C size}. As a consequence, every edge in $A\up \ell$ has at least $\tau_0\ab C/2$ extensions to vertices in $C\up \ell$. That is, $\tau_{C\up \ell}(\G\up \ell) \geq \tau_0/2$.

	We are thus in a position to apply \cref{lem:one step} to $\G\up \ell$, with parameters $\tau=\tau_0/2$ and $\largereal = (\log n)/(4\log \frac 1p)$, as defined above. In order to apply the lemma, we must check that $\largereal \geq 2^{k+1}/\tau=2^{k+2}/\tau_0$, which indeed holds by our assumption that $n\geq p^{-10^k/\tau_0}$. We obtain from \cref{lem:one step} a vertex $c^*$ and sets $A_i^*\subseteq A_i\up \ell$. We now set $C\up {\ell+1}=C\up \ell \setminus\{c^*\}$ and $A_i\up {\ell+1}=A^*_i$, and claim that this satisfies the desired properties.

	Property \ref{property:C size} is immediate since we removed one vertex from $C\up \ell$, as is property \ref{property:complete}, since the newly added vertex to $C \setminus C\up{\ell+1}$, namely $c^*$, is adjacent to all vertices in $A\up{\ell+1}_{1}\cup \dots \cup A_{k}\up{\ell+1}$ by the statement of \cref{lem:one step}. Moreover, we indeed have that $\A\up{\ell+1}$ is strictly $\largereal$-balanced, and satisfies
	\[
		\en_\largereal(\A\up{\ell+1}) \geq 2^{-2^{k+3}/\tau_0} \en_\largereal(\A\up \ell) \geq 2^{-2^{k+3}(\ell+1)/\tau_0} p^\largereal\prod_i\ab{A_i},
	\]
	proving \ref{property:energy}.

	Finally, we stop this process upon defining $\G\up s$. Let $C^\dagger = C \setminus C\up s, A_i^\dagger = A_i\up s$. By construction, we have $\ab{C^\dagger}=s $, proving \ref{it:Cdagger size}. We also automatically get \ref{it:daggers complete} from property \ref{property:complete}. It remains to prove \ref{it:Adagger size} and \ref{it:Adagger density}. For both of these, we note that
	\[
		2^{-2^{k+3}s/\tau_0} \geq 2^{-(\log n)/(4\log (1/p))} = 2^{-\largereal} \geq n^{-1/4},
	\]
	where the final inequality holds since $p\leq \frac 12$. Additionally, by our choice of $\largereal$, we see that $p^{\largereal}=n^{-1/4}$. Therefore,
	\[
		d(\A^\dagger)^\largereal\prod_i\ab{A_i^\dagger}=\en_\largereal(\A^\dagger) \geq 2^{-2^{k+3}s/\tau_0} p^\largereal \prod_i \ab{A_i} \geq n^{-1/4} \cdot n^{-1/4} \cdot \prod_i\ab{A_i},
	\]
	by property \ref{property:energy} and our choices of $s$ and $\largereal$. By combining this with the trivial upper bounds $\ab{A_i^\dagger}\leq \ab{A_i}$ and $d(\A^\dagger)$, we conclude that $\ab{A_i^\dagger}\geq \ab{A_i}/\sqrt n \geq \sqrt n$ for all $i$, proving \ref{it:Adagger size}. Similarly, we find that
	\[
		d(\A^\dagger)^\largereal \geq 2^{-\largereal}\cdot p^\largereal,
	\]
	which yields \ref{it:Adagger density}.
\end{proof}

As a corollary, we can immediately find a large complete tripartite graph in a graph with many triangles. We will eventually apply this as an intermediate step in our switching argument, so the statement of the following result is optimized with that application in mind. We denote by $\kappa_3(G)$ the triangle density of a tripartite graph $G$, that is, the total number of triangles divided by the product of the sizes of the three parts. 
\begin{corollary}\label{cor:good for switcher}
	Let $\kappa,\nu \in (0,\frac 12]$ be parameters, and let $m \geq \kappa^{-400/\nu}$.
	Let $G$ be a tripartite graph with parts $A,B,C$, where $\min \{\ab A , \ab B , \ab C\}\geq m$. If $\kappa_3(G)\geq \kappa$ and $\kappa_3(G)/d(A,B)\geq \nu$, then $G$ contains a copy of $K_{s,t,t}$, where
	\[
		s = \bflo{\frac{\nu \log m}{512\log \frac 1 \kappa}}  \qquad \text{and}\qquad t = \bflo{\frac{\log m}{12\log \frac1 \kappa}}.
	\]
\end{corollary}
\begin{proof}
	Let $\tau = \nu/2$, and let $F\subseteq E(A,B)$ denote the set of edges lying in at most $\tau\ab C$ triangles of $G$. Then let $G'$ be the subgraph of $G$ obtained by discarding the edges of $F$ (and not discarding any vertices); by construction, we have $\tau_C(G')\geq \tau$. 
	The total number of triangles using an edge in $F$ is at most 
	\[
		\tau \ab C \cdot e(A,B) =\tau d(A,B) \ab A \ab B \ab C \leq \frac \nu 2 \cdot \frac{\kappa_3(G)}{\nu}\ab A \ab B \ab C =\frac {\kappa_3(G)} 2 \ab A \ab B \ab C.
	\]
	Thus, at least half the triangles in $G$ are also present in $G'$. In particular,
	\[
		e_{G'}(A,B) \geq \frac{\kappa_3(G) \ab A \ab B \ab C}{2\ab C}  \geq \frac\kappa 2 \ab A \ab B,
	\]
	hence $d_{G'}(A,B)\geq \kappa/2\geq \kappa^2$, using our assumption $\kappa \leq \frac 12$. 

	We may now apply \cref{prop:basic iteration} with parameters $k=2, \tau_0=\tau, p = \kappa^2$ to $\G=G'$; note that a $2$-mixed graph is the same as a tripartite graph, and that $m$ is sufficiently large to apply \cref{prop:basic iteration} since $m\geq \kappa^{-400/\nu}= p^{-100/\tau}$. We obtain sets $A^\dagger \subseteq A, B^\dagger \subseteq B, C^\dagger \subseteq C$, satisfying
	\begin{enumerate}[label=(\alph*)]
		\item $\ab{C^\dagger} = \flo{(\tau\log m)/(2^7\log \frac 1 p)}=s$,
		\item $\ab{A^\dagger}\geq \sqrt m$ and $\ab{B^\dagger}\geq \sqrt m$,
		\item $d(A^\dagger,B^\dagger)\geq p/2$, and
		\item every vertex of $C^\dagger$ is adjacent to every vertex of $A^\dagger\cup B^\dagger$. 
	\end{enumerate}
	In particular, given the last point, it suffices to show that $(A^\dagger,B^\dagger)$ contains a copy of $K_{t,t}$. But this is an immediate consequence of \cref{thm:kst} when we note that
	\[
		\bflo{\frac{\log \min \{\ab{A^\dagger},\ab{B^\dagger}\}}{2\log \frac 2p}} \geq \bflo{\frac{\log \sqrt m}{2\log \frac{1}{\kappa^3}}} = \bflo{\frac{\log m}{12 \log \frac 1 \kappa}}=t.\qedhere
	\]
\end{proof}

\section{An improved bound via switching}\label{sec:switching}

In order to simplify the presentation of the argument, it is helpful to introduce the following definition, which precisely captures the kind of subgraph whose presence would be useful in the argument above.
\begin{definition}
	Let $H$ be a tripartite graph with parts $X,Y,Z$. We call $H$ an \emph{$(m,\mu)$-switcher} if $\min \{\ab X, \ab Y, \ab Z\}\geq m$,  $\kappa_3(H)\geq \mu^2$, and $\kappa_3(H)/d(X,Y)\geq \mu$.
\end{definition}
Note that this condition is monotone in both parameters, in the sense that every $(m,\mu)$-switcher is also an $(m',\mu')$-switcher for all $m' \leq m$ and all $\mu' \leq \mu$. 
From \cref{prop:basic iteration}, we see that if $\mu\leq \frac 12$ and $m$ is sufficiently large, then every $(m,\mu)$-switcher contains a copy of $K_{s,t,t}$, where $s = \Omega((\mu\log m)/(\log \frac 1 \mu))$ and $t = \Omega((\log m)/(\log \frac 1 \mu))$.

Using these ideas, we prove the following strengthening of \cref{lem:one step}. Note that we have added the assumption that no subgraph of $G$ is a switcher, and have obtained a stronger conclusion in \ref{switchit:energy bound}.
\begin{lemma}\label{lem:one step switching}
	Fix parameters $\tau \in (0,\frac 12]$ and $\largereal \geq 50/\tau$. Let $G$ be a tripartite graph with parts $(A,B,C)$, and suppose that $\tau_C(G)\geq \tau$. Let  $m \leq \frac \tau 4 \min \{\ab A, \ab B, \ab C\}$ be a positive integer, and let $p = d(A,B)$. Suppose further that $(A,B)$ is strictly $\largereal$-balanced, and that no subgraph of $G$ is an $(m, \sqrt{p\tau}/4)$-switcher. 

	Then there exist a vertex $c^* \in C$ and non-empty subsets $A^* \subseteq A, B^* \subseteq B$ satisfying the following properties.
	\begin{enumerate}[label=(\roman*)]
	\item $c^*$ is adjacent to all vertices in $A^*$ and $B^*$. \label{switchit:complete}
	\item We have \label{switchit:energy bound}
	\[
		\en_\largereal(A^*,B^*)\geq \tau^{20/\sqrt {p\tau}}\en_\largereal(A,B).
	\]
	\item $(A^*,B^*)$ is strictly $\largereal$-balanced.\label{switchit:balanced}
	\end{enumerate}
\end{lemma}
\begin{proof}
	Let $\kappa=\kappa_3(G)$. Note that the total number of triangles in $G$ equals $\kappa \ab A \ab B \ab C$, but is also at least $e(A,B) \cdot \tau \ab C$, since every edge between $A$ and $B$ lies in at least $\tau_C(G)\ab C\geq \tau\ab C$ triangles; this in particular implies that $\kappa \geq p\tau$.
	Let $C_0\subseteq C$ consist of all vertices $c \in C$ which lie in at most $\frac 14 \kappa \ab A \ab B$ triangles. The total number of triangles touching $C_0$ is at most $\frac 14 \kappa\ab A \ab B \ab C $, hence there are at least $\frac 34 \kappa \ab A \ab B \ab C$ triangles touching a vertex in $C \setminus C_0$. Each such vertex lies in at most $e(A,B)$ triangles, hence
	\[
		\ab{C \setminus C_0}\geq \frac{\frac 34 \kappa \ab A \ab B \ab C}{e(A,B)} \geq \frac 34 \cdot \frac{e(A,B)\cdot \tau \ab C}{e(A,B)} = 
		\frac {3\tau} 4 \ab C.
	\]
	Let $\mu=\sqrt {p\tau}/4$, and note that $\mu \leq \sqrt \kappa/4$ since $\kappa \geq p\tau$. Let us call a vertex $c \in C$ \emph{$A$-sparse} if $c \notin C_0$, and additionally $\ab{N(c)\cap A} \leq \mu \ab A$. Let $C_A$ denote the set of $A$-sparse vertices in $C$, and suppose that $\ab{C_A}\geq m$. Then consider the induced subgraph $G_A$ on vertex set $A \cup B \cup C_A$. Since no vertex of $C_A$ is in $C_0$, we see that $G_A$ contains at least $( \frac 1 4 \kappa \ab A \ab B ) \cdot \ab{C_A}$ triangles, and hence $\kappa_3(G_A)\geq \kappa/4>\mu^2$. Moreover, as every vertex in $C_A$ has degree at most $\mu\ab A$ to $A$, we see that $d(A,C_A)\leq \mu$, hence $\kappa_3(G_A)/d(A,C_A)>\mu^2/\mu=\mu$.
	Finally, we have $\min \{\ab A, \ab B, \ab{C_A}\}\geq m$, hence $G_A$ is an $(m,\mu)$-switcher on parts $X=A, Y=C_A, Z=B$. This contradicts our assumption that no subgraph of $G$ is an $(m,\mu)$-switcher, so we conclude that $\ab{C_A}<m$. By the same logic, we have $\ab{C_B}<m$, where $C_B$ is defined analogously.

	We conclude that $\ab{C \setminus(C_0 \cup C_A \cup C_B)}\geq (3\tau/4) \ab C - 2m \geq \tau \ab C/4$. Fix any vertex $c^* \in C \setminus (C_0 \cup C_A \cup C_B)$, and let $A^\oplus=N(c^*)\cap A, B^\oplus=N(c^*)\cap B$. By the choice of $c^*$, we know that $e(A^\oplus,B^\oplus)\geq \frac 14 \kappa \ab A \ab B$, and that $\ab{A^\oplus}\geq \mu \ab A, \ab{B^\oplus}\geq \mu \ab B$. 

	Note that if $p \leq 16\tau$, then also $p \leq 16(\kappa/p)$, since $\kappa \geq p\tau$. Therefore,
    \[
    \left(\frac \kappa p\right)^2 \geq \frac \kappa p \cdot \frac p{16} = \frac \kappa{16}\geq \mu^2
    \]
    implying that $\kappa_3(G)/d(A,B)=\kappa/p\geq \mu$.
	Since $\min \{\ab A, \ab B , \ab C\}\geq m$, we conclude that $G$ itself is an $(m,\mu)$-switcher, a contradiction. Hence we have $p> 16\tau$.

	Now let $\eta = \sqrt{\tau/p}$, and note by the above that $\eta \leq \sqrt{\tau/(16\tau)}=\frac 14$. Let $\alpha=({15\eta p \log \tfrac 1 \eta})/ ({\tau \largereal})$. Note that since $\largereal \geq 50/\tau$ and $\eta \leq \frac 14$, we have that $\alpha \leq p/2$. 
    We also note for future reference that $\eta \ab A \geq 4\eta/\tau>4/\sqrt \tau >5$, using that $1 \leq m \leq \frac \tau 4 \ab A$. This implies that $\flo{\eta \ab A}\geq \frac{4}{5}\eta \ab A$, and similarly $\flo{\eta\ab B}\geq \frac{4}{5}\eta \ab B$.

    Our key claim is that
	\begin{equation}\label{eq:lose alpha}
		d(A^\oplus,B^\oplus)\geq p-\alpha.
	\end{equation}
	Suppose for contradiction that $d(A^\oplus,B^\oplus)<p-\alpha<p$. In particular, we find that
	\[
		p\ab{A^\oplus}\ab{B^\oplus}>e(A^\oplus,B^\oplus) \geq \frac 14 \kappa \ab A \ab B \geq \frac 14 (p\tau)\ab A \ab B,
	\]
	since $\kappa \geq \tau_C(G)p\geq \tau p$. Rearranging, we conclude that $\ab{A^\oplus}\ab{B^\oplus} \geq \frac \tau 4 \ab A \ab B$.

	Let $A^+$ be a uniformly random subset of $A^\oplus$ of size $\min\{\ab{A^\oplus},\flo{\eta\ab A}\}$, and similarly for $B^+$. 
    By averaging, we may find a specific choice of $A^+,B^+$ such that $d(A^+,B^+) < p-\alpha$; let us fix such a choice. The key point is that in doing this, we have lost essentially nothing on the lower bound on the product of the sizes of $A^+$ and $B^+$, since
	\[
    \ab{A^+}\ab{B^+} = \min \{\ab{A^\oplus}\ab{B^\oplus}, \flo{\eta \ab A} \ab{B^\oplus}, \flo{\eta\ab B}\ab{A^\oplus}, \flo{\eta \ab A}\flo{\eta \ab B}\}. 
	\]
	The first term is at least $\frac \tau 4 \ab A \ab B$. 
    The next two terms are each at least $\frac{4}{5}\eta \mu \ab A \ab B = \frac \tau 5 \ab A \ab B$, by our choices of $\eta=\sqrt{\tau/p}$ and $\mu=\sqrt{p\tau}/4$. Finally, the last term is at least $(\frac{4}{5})^2\eta^2 \ab A \ab B > \frac 15{(\tau/p)\ab A \ab B} \geq \frac \tau5 \ab A \ab B$. So in total, we conclude that $\ab{A^+}\ab{B^+}\geq \frac \tau 5 \ab A \ab B$. 

	Let $A^- = A \setminus A^+, B^- = B \setminus B^+$. Since $\ab{A^+} \leq \eta \ab A, \ab{B^+}\leq \eta \ab B$, we see that for every $(X,Y) \in \{(A^-,B^-),(A^+,B^-),(A^-,B^+)\}$, we have that $(\ab X \ab Y)/(\ab A \ab B) \in [0,\eta] \cup [(1-\eta)^2,1]$. Moreover, by our choice of $\largereal$ and $\eta$, we have $\largereal \geq \frac{50}\tau \geq \log \frac 1 \tau \geq \log \frac 1 \eta$. Therefore, since $(A,B)$ is strictly $\largereal$-balanced, we may apply \cref{lemit:edge count strengthened} to conclude that for each such choice of $(X,Y)$, we have
	\[
		e(X,Y) \leq p \ab X \ab Y + (\eta \log \tfrac 1 \eta) \frac p \largereal \ab A \ab B \leq  p \ab X \ab Y + \frac {5\eta p \log \tfrac 1 \eta} {\tau \largereal} \ab{A^+}\ab{B^+}.
	\]
	Summing this inequality over all three choices of $(X,Y)$, as well as the inequality $e(A^+,B^+) < (p-\alpha)\ab{A^+}\ab{B^+}$, we conclude that
	\[
		e(A,B) < p \ab A \ab B + \left( \frac {15\eta p \log \tfrac 1 \eta} {\tau \largereal} -\alpha\right) \ab {A^+} \ab {B^+}.
	\]
	But this is impossible since, by our choice of $\alpha$, the second term above is $0$, and the first term equals $e(A,B)$. This contradiction concludes the proof of \eqref{eq:lose alpha}.

	It now only remains to estimate $\en_\largereal(A^\oplus, B^\oplus)$. For this, we recall that $e(A^\oplus, B^\oplus)\geq \frac 14 \kappa \ab A \ab B \geq \frac 14 p\tau \ab A \ab B$, and compute that
	\begin{align*}
		\en_\largereal(A^\oplus, B^\oplus) &= e(A^\oplus, B^\oplus) d(A^\oplus, B^\oplus)^{\largereal-1}\\
		&\geq \left( \frac 14 p\tau \ab A \ab B \right)(p-\alpha)^{\largereal-1}\\
		&=(\ab A \ab B p^\largereal) \left( \frac \tau4 p^{1-\largereal}(p-\alpha)^{\largereal-1} \right)\\
		&\geq \en_\largereal(A,B) \left( \tau^3 \left( 1- \frac \alpha p \right)^{\largereal-1} \right),
	\end{align*}
	using that $\tau \leq 1/2$ in the final inequality.
	Since $\alpha \leq p/2$, we may apply the inequality $1-\alpha/p \geq 2^{-2\alpha/p}$, which implies $(1-\alpha/p)^{\largereal-1} \geq 2^{-2\largereal \alpha/p}$. Now note that
	\[
		\frac{2\largereal \alpha}{p} = \frac{30 \eta \log \frac 1 \eta}{\tau} \leq 30 \frac{\sqrt{\tau/p}\cdot \log(1/\sqrt \tau)}{\tau} = \frac{15 \log \frac 1 \tau}{\sqrt{\tau p}} 
	\]
	where we use that $\eta = \sqrt{\tau/p}\geq \sqrt{\tau}$ in the first inequality. We may thus conclude that
	\[
		\tau^3 \left( 1- \frac \alpha p \right)^{\largereal-1} \geq \tau^3 \cdot 2^{-15\log(1/\tau)/\sqrt {p\tau}} = \tau^3 \cdot \tau^{15/\sqrt {p\tau}} \geq \tau^{20/\sqrt {p\tau}},
	\]
	and hence that $\en_\largereal(A^\oplus, B^\oplus) \geq \tau^{20/\sqrt {p\tau}}\en_\largereal(A,B)$. Finally, we pass to a strictly $\largereal$-balanced subgraph $(A^*,B^*)$ satisfying $\en_\largereal(A^*,B^*)\geq \en_\largereal(A^\oplus, B^\oplus)$, as given by \cref{lem:pass to max}, which satisfies the desired properties since $c^*$ is adjacent to all vertices in $A^\oplus \cup B^\oplus \supseteq A^*\cup B^*$. 
\end{proof}

Our next result is obtained by iterating \cref{lem:one step switching}, and the proof is very similar to that of \cref{prop:basic iteration}.

\begin{proposition}\label{prop:switch iteration}
	Fix parameters $\tau_0,p_0 \in (0,\frac 12]$. Let $n \geq p_0^{-400/\tau_0}$ and let $G$ be a tripartite graph with parts $A,B,C$ satisfying $\min \{\ab A,\ab B, \ab C \}\geq n$. Suppose that $\tau_C(G)\geq \tau_0$ and $d(A,B)\geq p_0$, and that no subgraph of $G$ is a $(\flo{n^{1/3}}, \sqrt{p_0\tau_0}/8)$-switcher. 

	Then there exist subsets $A^\dagger \subseteq A, B^\dagger \subseteq B, C^\dagger \subseteq C$ satisfying
	\begin{enumerate}[label=(\alph*)]
		\item $\ab{C^\dagger}=\flo{({\sqrt{p_0\tau_0}\log n})/({320\log \frac 1{p_0} \log \frac 1{\tau_0}})}$,
		\item $\ab{A^\dagger}\geq \sqrt n$ and $\ab{B^\dagger}\geq \sqrt n$,
		\item $d(A^\dagger, B^\dagger)\geq p_0/2$, and
		\item Every vertex of $C^\dagger$ is adjacent to every vertex of $A^\dagger \cup B^\dagger$. 
	\end{enumerate}
\end{proposition}
As the proof of \cref{prop:switch iteration} is quite similar to that of \cref{prop:basic iteration}, we will be brief, mostly highlighting the places where the differences lie.
\begin{proof}[Proof of \cref{prop:switch iteration}]
	Let
	\[
		s = \bflo{\frac{\sqrt{p_0\tau_0}\log n}{320\log \frac 1{p_0} \log \frac 1{\tau_0}}}, \qquad \largereal = \frac{\log n}{4 \log \frac 1{p_0}}, \qquad m = \flo{n^{1/3}}, \qquad\kappa_0 = \frac{p_0\tau_0}4, \qquad \mu = \frac{\sqrt{\kappa_0}}{4}.
	\]
	We note that by our assumption $n \geq p_0^{-400/\tau_0}\geq 2^{400/\tau_0}$, we have that $\frac {\tau_0}8\sqrt n \geq m$. 
	We will define a nested sequence of induced subgraphs $G\supseteq G\up 0 \supseteq \dots \supseteq G\up s$ on vertex sets $(A\up \ell, B\up \ell, C\up \ell)_{\ell=0}^s$, satisfying that $\ab{C\up \ell}=\ab C-\ell$, that every vertex in $C \setminus C\up \ell$ is adjacent to all of $A\up \ell \cup B \up \ell$, and that $(A\up \ell, B\up \ell)$ is strictly $\largereal$-balanced and satisfies
	\[
		\en_\largereal(A\up \ell, B\up \ell) \geq \tau_0^{40\ell/\sqrt{\kappa_0}} p_0^\largereal \ab A \ab B.
	\]
	To begin the iteration, we let $C\up 0=C$, and let $(A\up 0, B\up 0)$ be a strictly $\largereal$-balanced subgraph of $(A,B)$, as given by \cref{lem:pass to max}, which trivially satisfies these properties. 

	Now suppose we have defined $G\up \ell$ for some $\ell\leq s$. We note several properties of this graph. First, as $\ab{C \setminus C\up \ell} = \ell \leq s \leq \tau_0 n/2$, we have that $\tau_{C\up \ell}(G\up \ell)\geq \tau_0/2$. Second, using the fact that $p_0^\largereal=n^{-1/4}$ and 
	\[
		\tau_0^{40\ell/\sqrt{\kappa_0}} \geq \tau_0^{40s/\sqrt{\kappa_0}}\geq 2^{-\largereal} \geq n^{-1/4},
	\]
	we find that
	\[
		d(A\up \ell,B\up \ell)^\largereal\ab{A\up \ell}\ab{B\up \ell} =\en_\largereal(A\up \ell, B\up \ell) \geq 2^{-\largereal}p_0^\largereal \ab A \ab B \geq n^{-1/2}\ab A \ab B.
	\]
	This implies immediately that $\ab{A\up \ell}\geq \sqrt n, \ab{B\up \ell}\geq \sqrt n$, as well as that $d(A\up \ell,B\up \ell)\geq p_0/2$. 

	Assuming now that $\ell<s$,
	we apply \cref{lem:one step switching} to $G\up \ell$, with parameters $\tau=\tau_0/2, p=d(A\up \ell, B\up \ell) \geq p_0/2$ and $\largereal,m$ defined as above. Our assumption $n \geq p_0^{-400/\tau_0}$ implies that $\largereal \geq 100/\tau_0 = 50/\tau$. We've also seen that $m \leq \frac {\tau_0} 8 \sqrt n\leq \frac \tau 4 \min \{\ab{A\up \ell},\ab{B\up \ell},\ab{C\up \ell}\}$, using that $\ab{C\up \ell} \geq\ab C-s \geq \sqrt n$. Finally, since $G\up \ell$ is a subgraph of $G$, we see that no subgraph of $G\up \ell$ is an $(m,\sqrt{p_0 \tau_0}/8)$-switcher, hence no subgraph of $G\up \ell$ is an $(m,\sqrt{p\tau}/4)$-switcher. We have thus verified all the conditions of \cref{lem:one step switching}, which produces for us a vertex $c^*$ and sets $A^* \subseteq A\up \ell, B^* \subseteq B\up \ell$ satisfying
    \[
    \en_\largereal(A^*,B^*)\geq \tau^{20/\sqrt{p\tau}} \en_\largereal(A\up \ell,B\up \ell)\geq \left(\frac{\tau_0} 2\right)^{20/\sqrt{p_0\tau_0/4}}\en_\largereal(A\up \ell,B\up\ell) \geq \tau_0^{40/\sqrt{\kappa_0}}\en_\largereal(A\up \ell,B\up \ell).
    \]
    We set $C\up {\ell+1}=C\up \ell \setminus \{c^*\}, A\up{\ell+1}=A^*, B\up {\ell+1}=B^*$, and see that all of our desired properties remain true for $G\up{\ell+1}$. Moreover, when we stop the iteration at step $s$, the computations above show that we may take $A^ \dagger=A\up s, B^\dagger=B\up s, C^\dagger=C\setminus C\up s$, which satisfy all the desired properties. 
\end{proof}

\section{Bounds in Nikiforov's theorem}\label{sec:main proof}
In this section, we use the results proved in the previous sections to prove \cref{thm:main}.
In order to apply \cref{prop:basic iteration,prop:switch iteration}, we need a result that converts an arbitrary graph with many copies of $K_{k+1}$ into a $k$-mixed graph in which we control both $\tau_C(G)$ and $d(\mathcal A)$. This is not too hard to do via standard tools from the theory of graph regularity, and we prove the following result in \cref{sec:appendix}.
\begin{lemma}\label{lem:regularize-cliques}
	For every integer $k \geq 2$ and every $\gamma\in (0,\frac 12]$, there exists some $\zeta \geq \gamma^{k^{50}\gamma^{-10}}>0$ and some $\tau \in [\gamma^{2/(k+1)},\frac 12]$ such that the following holds. Let $G$ be an $N$-vertex graph with at least $\gamma N^{k+1}$ copies of $K_{k+1}$.  Then there exist disjoint sets $A_1,\dots,A_k,C \subseteq V(G)$, each of size at least $\flo{\zeta N}$, as well as a $k$-mixed graph $\G$ on vertex set $(\A,C)$ such that
	\begin{enumerate}
		\item Every edge of $E_\G(C,A_1 \cup \dotsb \cup A_k)$ is an edge of $G$,
		\item Every edge of $E_\G(\A)$ is a copy of $K_k$ in $G$, and
		\item $\tau_C(\G)\geq \tau \geq \gamma^{2/(k+1)}$ and $d(\A)\tau\geq \gamma/4$.
	\end{enumerate}
\end{lemma}

We begin by proving \cref{thm:main} in the case of triangles.
\begin{theorem}\label{thm:main triangle}
	$c_{K_3}(\gamma)=\Omega(\sqrt \gamma/(\log \frac 1 \gamma)^2)$. 
\end{theorem}
\begin{proof}
    We may assume that $\gamma \leq \frac 12$, by choosing the implicit constant large enough. 
	Let $N \geq \gamma^{-2^{60}\gamma^{-10}}$ be sufficiently large, and let $G_0$ be an $N$-vertex graph with at least $\gamma N^3$ triangles. Applying \cref{lem:regularize-cliques} with $k=2$, we obtain a tripartite subgraph $G$ on parts $A,B,C \subseteq V(G)$, satisfying $n=\min\{\ab A, \ab B, \ab C\}\geq \flo{\zeta N}\geq \sqrt N$, as well as some $\tau \in [\gamma^{2/3},\frac 12]$ such that $\tau_C(G)\geq \tau \geq \gamma^{2/3}$ and $d_{G}(A,B)\tau\geq \gamma/4$. This in particular implies $d_G(A,B)\geq \gamma/4$. Note that $n \geq \sqrt N \geq (\gamma/4)^{-400/\gamma^{2/3}}$. 

	We split into two cases. First, suppose that no subgraph of $G$ is a $(\flo{n^{1/3}}, \sqrt\gamma/16)$-switcher. Then we are precisely in the setting of \cref{prop:switch iteration} (with parameters $\tau_0=\tau$ and $p_0= \gamma/(4\tau_0)$), hence we may find sets $A^\dagger, B^\dagger, C^\dagger$ as in the statement of \cref{prop:switch iteration}. We have that
	\[
		s=\ab{C^\dagger} \geq \bflo{\frac{\sqrt{\gamma/4}\log n}{320 \log \frac 1 {\gamma^{2/3}} \log \frac 4{\gamma}}} \geq \bflo{\frac{\sqrt{\gamma}\log n}{2000(\log \frac 1 \gamma)^2}} \geq \bflo{\frac{\sqrt{\gamma}\log N}{4000(\log \frac 1 \gamma)^2}},
	\]
	as well as $\ab{A^\dagger},\ab{B^\dagger}\geq \sqrt n \geq N^{1/4}$, and all vertices of $C^\dagger$ are adjacent to all vertices of $A^\dagger \cup B^\dagger$. We now apply \cref{thm:kst} to the bipartite graph between $A^\dagger$ and $B^\dagger$, which has edge density at least $d(A,B)/2\geq \gamma/8$. We conclude that $(A^\dagger,B^\dagger)$ contains a copy of $K_{t,t}$ for 
	\[
		t = \bflo{\frac{\log N^{1/4}}{2\log \frac 8 \gamma}}\geq \bflo{\frac{\log N}{32 \log \frac 1 \gamma}}\geq s.
	\]
	Together with $C^\dagger$, we find a copy of $K_{s,t,t}\supseteq K_{s,s,s}$ in $G$, completing the proof in this case.

	On the other hand, if $G$ contains a $(\flo{n^{1/3}},\sqrt \gamma/16)$-switcher, we apply \cref{cor:good for switcher} to this switcher, with parameters $m=\flo{n^{1/3}}\geq N^{1/7},\kappa = \gamma/256$, and $\nu = \sqrt \gamma/16$, noting that $m \geq \kappa^{-400/\nu}$. We conclude that $G$ contains a complete tripartite graph $K_{s',t',t'}$, where
	\[
		s' = \bflo{\frac{\nu \log m}{512 \log \frac{256}{\gamma}}} \geq \bflo{\frac{\sqrt \gamma \log N}{10^7 \log \frac 1 \gamma}}, \qquad t' = \bflo{\frac{\log m}{12\log \frac {256} \gamma}} \geq s',
	\]
	and we again obtain the desired result. 
\end{proof}
Next, we use \cref{prop:basic iteration} to provide a recursive bound on $c_{K_{k+1}}(\gamma)$ in terms of $c_{K_k}(\gamma)$. 
\begin{theorem}\label{thm:Kk+1 reduction}
	For every $k \geq 3$, we have
	\[
		c_{K_{k+1}}(\gamma) \geq \min \left\{\Omega_k\left( \frac{\gamma^{2/(k+1)}}{\log \frac 1 \gamma} \right), \frac 18 c_{K_k}\left( \frac\gamma{8k^k} \right)\right\}.
	\]
\end{theorem}
\begin{proof}
    As before, we may assume that $\gamma \leq 2^{-k}$, by picking the implicit constant sufficiently large.
	Let $\gamma'=\gamma/(8k^k)$, and let $m_0$ be chosen sufficiently large so that for all $m \geq m_0$, every $m$-vertex graph with $\gamma'm^k$ copies of $K_k$ contains a copy of $K_k[t]$, where $t\geq \frac 12 c_{K_k}(\gamma') \log m$. 

	Let now $N$ be sufficiently large, and in particular at least $m_0^4$. Let $G$ be an $N$-vertex graph with at least $\gamma N^{k+1}$ copies of $K_{k+1}$. We apply \cref{lem:regularize-cliques} to find disjoint sets $A_1,\dots,A_k,C\subseteq V(G)$, as well as a $k$-mixed graph $\G$ on this vertex set. Since $N$ is sufficiently large, we may assume $n = \min\{\ab{A_1},\dots,\ab{A_k},\ab C\}\geq \sqrt N$. We now apply \cref{prop:basic iteration} to $\G$, with parameters $\tau_0=\gamma^{2/(k+1)}$ and $p = \gamma/(4\tau_C(\G))\geq \gamma/4$, to obtain sets $A_i^\dagger, C^\dagger$. We have that 
	\[
		\ab{C^\dagger} = \bflo{\frac{\tau_0\log n}{2^{k+5}\log \frac 1 p}} \geq \bflo{\frac{\gamma^{2/(k+1)}\log N}{2^{k+8}\log \frac 1 \gamma}}.
	\]
	Now, let $A_i' \subseteq A_i^\dagger$ be a randomly chosen set of size exactly $\flo{N^{1/4}}$, which exists since $\ab{A_i^\dagger}\geq \sqrt n \geq \flo{N^{1/4}}$. In expectation, the number of $k$-uniform edges of $\G$ in $A_1'\cup \dots \cup A_k'$ is at least $(p/2)\flo{N^{1/4}}^k$, hence we may fix an outcome with at least that many edges. This yields for us an induced subgraph of $G[A_1^\dagger \cup \dots \cup A_k^\dagger]$ of order $m=k\flo{N^{1/4}}$ with at least $(\gamma/8)\flo{N^{1/4}}^k = \gamma' m^k$ copies of $K_k$. By the definition of $m_0$, we can find in this subgraph a copy of $K_k[t]$, for $t = \frac 12 c_{K_k}(\gamma') \log m \geq \frac 18 c_{K_k}(\gamma')\log N$, and all of its vertices are adjacent in $G$ to $C^\dagger$. Hence, together with $C^\dagger$, we obtain a blowup of $K_{k+1}$ of order $\min \{t, \ab{C^\dagger}\}$, which yields the desired result. 
\end{proof}
We note that this proof actually shows that the ``sufficiently large'' condition on the number of vertices for $K_{k+1}$ is polynomial in that for $K_k$. In particular, since \cref{thm:main triangle} already applies for $N \geq \gamma^{-O(\gamma^{-10})}$, we conclude that similarly \cref{thm:Kk+1 reduction} also holds for $N \geq \gamma^{-O_k(\gamma^{-10})}$. 

Finally, we use the following simple and well-known argument, that shows that a bound on $c_{K_h}(\gamma)$ implies essentially the same bound on $c_H(\gamma)$ for any $h$-vertex graph $H$. 
\begin{proposition}\label{prop:monotonicity reduction}
	If $H$ has $h$ vertices, then $c_{H}(\gamma)\geq c_{K_h}(h^{-h}\gamma)$. 
\end{proposition}
\begin{proof}
	Let $G$ be an $N$-vertex graph, and suppose that $G$ contains at least $\gamma N^h$ copies of $H$. In particular, if we identify the vertex set of $H$ with $[h]$, then $G$ contains at least $\gamma N^h$ labeled copies of $H$. We now pick a random partition of $V(G)$ into $h$ parts $V_1,\dotsc,V_h$. Every labeled copy of $H$ has a probability $h^{-h}$ of being \emph{spanning}, that is, having its $i$th vertex end up in $V_i$ for all $i \in [h]$. So by linearity of expectation, there exists some $h$-partition of $V$ so that at least $h^{-h} \gamma N^h$ labeled copies of $H$ are spanning.

	We now define a new $h$-partite graph $\wt G$ on the same vertex set as $G$. We obtain $\wt G$ from $G$ by first deleting every edge contained in some $V_i$. Next, for every pair $(i,j)$ such that $ij \notin E(H)$, we place a complete bipartite graph between $V_i$ and $V_j$. However, if $ij \in E(H)$, then we let $\wt G$ have the same edges between $V_i$ and $V_j$ as are found in $G$.

    Every spanning copy of $H$ in $G$ thus yields a copy of $K_h$ in $\wt G$, and conversely, every copy of $K_h$ in $\wt G$ arises from a spanning copy of $H$ in $G$.
    So by the definition of $c_{K_h}(h^{-h} \gamma)$, we see that $\wt G$ must contain $K_h[t]$ as a subgraph, where $t \geq (c_{K_h}(h^{-h} \gamma)-o(1))\log N$, and where the $o(1)$ term tends to $0$ as $N\to \infty$. As $\wt G$ is $h$-partite, each of the parts of this $K_h[t]$ must lie in a distinct part $V_i$ of $\wt G$. But if we only restrict our attention to pairs $(V_i,V_j)$ where $ij \in E(H)$, we see that this $K_h[t]$ gives rise to a copy of $H[t]$ inside $G$. By the definition of $c_{H}(\gamma)$, this completes the proof.
\end{proof}

With these tools, the proof of \cref{thm:main} is essentially immediate.
\begin{proof}[Proof of \cref{thm:main}]
    The case $h=1$ is trivial, and the case $h=2$ is covered by \cref{thm:kst}, hence we may assume that $h \geq 3$. As before, we may also assume that $\gamma \leq \frac 12$.
	We first claim that for all $h \geq 3$, we have $c_{K_h}(\gamma)=\Omega_h(\sqrt \gamma/(\log \frac 1 \gamma)^2)$. We prove this by induction on $h$, with the base case $h=3$ being precisely the statement of \cref{thm:main triangle}. Having proved the statement for $h=k$, we obtain it for $h=k+1\geq 4$ by \cref{thm:Kk+1 reduction}, since
	\[
		c_{K_{k+1}}(\gamma)\geq \min \left\{\Omega\left( \frac{\gamma^{2/(k+1)}}{\log \frac 1 \gamma} \right), \frac 18 c_{K_k}\left( \frac\gamma{8k^k} \right)\right\} = \min \left\{ \Omega \left( \frac{\gamma^{2/4}}{\log \frac 1 \gamma}\right), \Omega_k\left( \frac{\sqrt \gamma}{(\log \frac 1 \gamma)^2} \right)  \right\},
	\]
	and both terms are at least $\Omega_{k+1}(\sqrt \gamma/(\log \frac 1 \gamma)^2)$, completing the induction. Finally, for a general graph $H$ with $h$ vertices, we see by \cref{prop:monotonicity reduction} that
	\[
		c_H(\gamma)\geq c_{K_h}(h^{-h}\gamma) = \Omega_H \left( \frac{\sqrt \gamma}{(\log \frac 1 \gamma)^2} \right) .\qedhere
	\]
\end{proof}

\section{Concluding remarks}\label{sec:conclusion}
As discussed in the introduction, our iterative approach to Nikiforov's theorem is closely connected to techniques used in many other problems, and this suggests that these connections should be explored further. For example, the resemblance between our technique and the one used to bound multicolor Ramsey numbers \cite{MR5057671} suggests that it may be useful to attack a multicolor Ramsey-theoretic variant of Nikiforov's theorem proposed in \cite[Conjecture 1.6]{MR5007509}. Moreover, in the setting of regularity lemmas, such iterative proofs can often be ``reverse engineered'' to show that the bound they give is optimal; this was first achieved in the breakthrough work of Gowers \cite{MR1445389}, and his technique has been adapted and extended to a wide variety of problems, e.g.\ \cite{MR5035093,MR5024850,MR4634684,MR4025519,MR3939561,MR3737374,MR3530502,MR3516883,MR3048552,MR2989432,MR2153903}. Although we have not succeeded in doing the analogous thing here, it is conceivable that by reverse-engineering our proof, one could show that, for example, $c_{K_3}(\gamma)\leq \gamma^{\Omega(1)}$, demonstrating that random graphs are very far from extremal in Nikiforov's theorem. 

Indeed, the most interesting question about this topic remains open, namely whether the true behavior of $c_H(\gamma)$ is polynomial or sub-polynomial. In case $H$ is triangle-free, the result of \cite{MR5007509} shows that the true value is $c_H(\gamma)=\Theta_H(1/{\log \frac 1 \gamma})$, but for any other $H$ there is a wide gap between the bounds
\[
	\gamma^{1/2+o(1)} \leq c_H(\gamma) \leq O_H\left( \frac{1}{\log \frac 1 \gamma} \right),
\]
where the lower bound is \cref{thm:main} and the upper bound comes from considering a random graph. Given the result of \cite{MR5007509}, it is natural to conjecture that the upper bound is closer to the truth, but we are not sure whether to believe this conjecture: as discussed above, it is natural to try to ``reverse-engineer'' our iterative proof to obtain a polynomial upper bound $c_H(\gamma)\leq \gamma^{\Omega(1)}$. While we have been unsuccessful in doing this, we believe this is a problem that deserves further attention.

Another promising direction is to take advantage of the apparent similarity between our technique and the book algorithm \cite{MR5067166}. Recall that the main difficulty we encounter is that when we select a vertex $c^* \in C$ and define $A^+=N(c^*)\cap A, B^+=N(c^*)\cap B$, the density $d(A^+,B^+)$ may be smaller than $d(A,B)$. This is the main effect hurting us in our iterative algorithm, and much of the argument is about trying to limit its harm. 

This is essentially the same issue one encounters when trying to prove upper bounds on diagonal Ramsey numbers with the book algorithm. Note that this problem can be described as a ``negative correlation'' phenomenon between the neighborhoods of different vertices. Indeed, the only way for $d(A^+,B^+)$ to be substantially smaller than $d(A,B)$ is if the neighborhoods of vertices in $C$ are biased away from the edges in $E(A,B)$. In \cite{MR5057671}, Balister et al.\ introduced a novel way of controlling the effect of such negative correlations, via their so-called \emph{geometric lemma}. Roughly speaking, they show that if there are a lot of these negative correlations, one is able to focus on a large subgraph where one actually has substantial \emph{positive} correlation. It would be very interesting if one could adapt their technique to our setting, and ideally to thus obtain a better bound on $c_H(\gamma)$. 

Finally, we note that, like most prior works on this subject \cite{MR2409174,MR2398823,MR2993136,MR4195582} (but notably not \cite{MR5007509}), \cref{thm:main} assumes that $N$ is sufficiently large with respect to $\gamma$. As discussed in detail in \cite[Section 1.2]{MR5007509}, it would also be of great interest to prove such results without this assumption, i.e.\ to allow $\gamma$ to tend to $0$ as a function of $N$. Our proof shows that \cref{thm:main} holds when $N \geq 2^{C_H \gamma^{-11}}$ for some constant $C_H$ depending on $H$, or equivalently that one can let $\gamma$ tend to $0$ as any function slower than $(\log N)^{-1/11}$ and still obtain meaningful results. This is unfortunately too slow for most applications, e.g.\ \cite[Problem 3]{MR2993136} or \cite[Conjecture 1]{MR2959395} (and in any case the bound of $\gamma^{1/2+o(1)}$ is too weak for these applications). We refer to \cite[Section 1.2]{MR5007509} for more information on this point.


\begin{thebibliography}{10}
\providecommand{\url}[1]{\texttt{#1}}
\providecommand{\urlprefix}{URL }
\providecommand{\eprint}[2][]{\url{#2}}

\bibitem{MR5057671}
P.~Balister, B.~Bollob\'as, M.~Campos, S.~Griffiths, E.~Hurley, R.~Morris, J.~Sahasrabudhe, and M.~Tiba, Upper bounds for multicolour {R}amsey numbers, \emph{J. Amer. Math. Soc.} \textbf{39} (2026), 765--780.

\bibitem{MR335347}
B.~Bollob\'as and P.~Erd\H{o}s, On the structure of edge graphs, \emph{Bull. London Math. Soc.} \textbf{5} (1973), 317--321.

\bibitem{MR485528}
B.~Bollob\'as, P.~Erd\H{o}s, and M.~Simonovits, On the structure of edge graphs. {II}, \emph{J. London Math. Soc. (2)} \textbf{12} (1975/76), 219--224.

\bibitem{MR5067166}
M.~Campos, S.~Griffiths, R.~Morris, and J.~Sahasrabudhe, An exponential improvement for diagonal {R}amsey, \emph{Ann. of Math. (2)} \textbf{203} (2026), 869--932.

\bibitem{MR2989432}
D.~Conlon and J.~Fox, Bounds for graph regularity and removal lemmas, \emph{Geom. Funct. Anal.} \textbf{22} (2012), 1191--1256.

\bibitem{MR2959395}
D.~Conlon, J.~Fox, and B.~Sudakov, {Erd\H{o}s}--{H}ajnal-type theorems in hypergraphs, \emph{J. Combin. Theory Ser. B} \textbf{102} (2012), 1142--1154.

\bibitem{MR1333857}
R.~A. Duke, H.~Lefmann, and V.~R\"{o}dl, A fast approximation algorithm for computing the frequencies of subgraphs in a given graph, \emph{SIAM J. Comput.} \textbf{24} (1995), 598--620.

\bibitem{MR928742}
P.~Erd\H{o}s, M.~Goldberg, J.~Pach, and J.~Spencer, Cutting a graph into two dissimilar halves, \emph{J. Graph Theory} \textbf{12} (1988), 121--131.

\bibitem{MR0018807}
P.~Erd\"{o}s and A.~H. Stone, On the structure of linear graphs, \emph{Bull. Amer. Math. Soc.} \textbf{52} (1946), 1087--1091.

\bibitem{MR3737374}
J.~Fox and L.~M. Lov\'asz, A tight lower bound for {S}zemer\'edi's regularity lemma, \emph{Combinatorica} \textbf{37} (2017), 911--951.

\bibitem{MR4195582}
J.~Fox, S.~Luo, and Y.~Wigderson, Extremal and {R}amsey results on graph blowups, \emph{J. Comb.} \textbf{12} (2021), 1--15.

\bibitem{MR4634684}
J.~Fox, H.~T. Pham, and Y.~Zhao, Tower-type bounds for {R}oth's theorem with popular differences, \emph{J. Eur. Math. Soc. (JEMS)} \textbf{25} (2023), 3795--3831.

\bibitem{MR5024850}
F.~Garbe and J.~Hladk\'y, A tower lower bound for the degree relaxation of the regularity lemma, \emph{Comb. Theory} \textbf{5} (2025), Paper No. 8, 17pp.

\bibitem{MR5007509}
A.~Gir\~ao, Z.~Hunter, and Y.~Wigderson, Blowups of triangle-free graphs, \emph{Adv. Comb.}  (2025), Paper No. 10, 23pp.

\bibitem{MR5035093}
L.~Gishboliner, A.~Shapira, and Y.~Wigderson, Is it easy to regularize a hypergraph with easy links?, \emph{Int. Math. Res. Not. IMRN}  (2026), Paper No. rnag018, 18pp.

\bibitem{MR1445389}
W.~T. Gowers, Lower bounds of tower type for {S}zemer\'edi's uniformity lemma, \emph{Geom. Funct. Anal.} \textbf{7} (1997), 322--337.

\bibitem{MR2153903}
B.~Green, A {S}zemer\'edi-type regularity lemma in abelian groups, with applications, \emph{Geom. Funct. Anal.} \textbf{15} (2005), 340--376.

\bibitem{2407.19026}
P.~Gupta, N.~Ndiaye, S.~Norin, and L.~Wei, Optimizing the {CGMS} upper bound on {R}amsey numbers, 2024. Preprint available at arXiv:2407.19026.

\bibitem{MR3530502}
K.~Hosseini, S.~Lovett, G.~Moshkovitz, and A.~Shapira, An improved lower bound for arithmetic regularity, \emph{Math. Proc. Cambridge Philos. Soc.} \textbf{161} (2016), 193--197.

\bibitem{MR3048552}
S.~Kalyanasundaram and A.~Shapira, A {W}owzer-type lower bound for the strong regularity lemma, \emph{Proc. Lond. Math. Soc. (3)} \textbf{106} (2013), 621--649.

\bibitem{MR0065617}
T.~K\"{o}vari, V.~T. S\'{o}s, and P.~Tur\'{a}n, On a problem of {K}. {Z}arankiewicz, \emph{Colloq. Math.} \textbf{3} (1954), 50--57.

\bibitem{2601.05221}
R.~Morris, Some recent results in {R}amsey theory, 2026. Preprint available at arXiv:2601.05221.

\bibitem{MR3516883}
G.~Moshkovitz and A.~Shapira, A short proof of {G}owers' lower bound for the regularity lemma, \emph{Combinatorica} \textbf{36} (2016), 187--194.

\bibitem{MR3939561}
G.~Moshkovitz and A.~Shapira, A sparse regular approximation lemma, \emph{Trans. Amer. Math. Soc.} \textbf{371} (2019), 6779--6814.

\bibitem{MR4025519}
G.~Moshkovitz and A.~Shapira, A tight bound for hypergraph regularity, \emph{Geom. Funct. Anal.} \textbf{29} (2019), 1531--1578.

\bibitem{MR2398823}
V.~Nikiforov, Graphs with many copies of a given subgraph, \emph{Electron. J. Combin.} \textbf{15} (2008), Note 6, 6pp.

\bibitem{MR2409174}
V.~Nikiforov, Graphs with many {$r$}-cliques have large complete {$r$}-partite subgraphs, \emph{Bull. Lond. Math. Soc.} \textbf{40} (2008), 23--25.

\bibitem{MR2993136}
V.~R\"{o}dl and M.~Schacht, Complete partite subgraphs in dense hypergraphs, \emph{Random Structures Algorithms} \textbf{41} (2012), 557--573.

\bibitem{MR2718688}
A.~Shapira and R.~Yuster, On the density of a graph and its blowup, \emph{J. Combin. Theory Ser. B} \textbf{100} (2010), 704--719.

\bibitem{MR4149162}
V.~Souza, Blowup {R}amsey numbers, \emph{European J. Combin.} \textbf{92} (2021), Paper No. 103238, 14pp.

\bibitem{MR540024}
E.~Szemer\'edi, Regular partitions of graphs, in \emph{Probl\`emes combinatoires et th\'eorie des graphes ({C}olloq. {I}nternat. {CNRS}, {U}niv. {O}rsay, {O}rsay, 1976)}, \emph{Colloq. Internat. CNRS}, vol. 260, CNRS, Paris, 1978,  399--401.

\bibitem{bourbaki}
Y.~Wigderson, Upper bounds on diagonal {R}amsey numbers [after {C}ampos, {G}riffiths, {M}orris, and {S}ahasrabudhe], \emph{Ast\'erisque} \textbf{462} (2025), 85--138.

\end{thebibliography}

\appendix
\section{Proof of Lemma \ref{lem:regularize-cliques}}\label{sec:appendix}
In this section, we present the proof of \cref{lem:regularize-cliques}. Before doing so, we recall a number of concepts and results related to Szemer\'edi's regularity lemma. Given $\varepsilon \in (0,1)$, a pair of vertex sets $(S,T)$ (not necessarily disjoint) in a graph $F$ is said to be $\varepsilon$-regular if $\ab{d(S',T')-d(S,T)} <\varepsilon$ for all $S' \subseteq S, T' \subseteq T$ with $\ab{S'} \geq \varepsilon \ab S, \ab{T'} \geq \varepsilon \ab T$. Suppose now that $F$ is $r$-partite, with $r$-partition $V=V_1 \sqcup \dotsb \sqcup V_r$. A \emph{cylinder} $K$ is a set of the form $W_1 \times \dotsb W_r$, where $W_i \subseteq V_i$ for all $i \in [r]$. For such a cylinder $K$, let $V_i(K)=W_i$. We say that $K$ is \emph{$\varepsilon$-regular} if $(W_i,W_j)$ is $\varepsilon$-regular for all $1 \leq i<j \leq r$. A \emph{cylinder partition} $\K$ is a partition of $V_1 \times \dotsb \times V_r$ into cylinders, and we say that $\K$ is \emph{$\varepsilon$-regular} if at most an $\varepsilon$-fraction of the $r$-tuples $(v_1,\ldots,v_r) \in V_1 \times \dotsb \times V_r$ are not in $\varepsilon$-regular cylinders of $\K$. The following weak regularity lemma of Duke, Lefmann, and R\"odl says that every $r$-partite graph has an $\varepsilon$-regular cylinder partition with a bounded number of parts. 
\begin{lemma}[Duke--Lefmann--R\"odl]\label{lem:dlr}
	For every $\varepsilon \in (0,\frac 12)$ and integer $r \geq 2$, there exists $\beta=\beta(\varepsilon,r) >0$ such that the following holds. Suppose that $F=(V,E)$ is an $r$-partite graph with an $r$-partition $V_1 \sqcup \dotsb \sqcup V_r$. Then there exists an $\varepsilon$-regular cylinder partition $\K$ of $V_1 \times \dotsb \times V_r$ such that $|V_i(K)|\geq \beta |V_i|$ for each $K \in \K$ and $i \in [r]$. Moreover, we may take $\beta \geq \varepsilon^{r^2 \varepsilon^{-5}}$. 
\end{lemma}
We will also need the following standard counting lemma.
\begin{lemma} \label{lem:counting}
	Let $H$ be a graph on vertex set $[t]$. Suppose $V_1,\ldots,V_{t}$ are (not necessarily distinct) vertex sets in a graph $G$ such that $(V_i,V_j)$ is $\varepsilon$-regular for all $\{i,j\} \in E(H)$. Then the number of tuples $(v_1,\ldots,v_t) \in V_1 \times \dotsb \times V_t$ such that $\{v_i,v_j\} \in E(G)$ whenever $\{i,j\} \in E(H)$ is
	\[
		\left( \prod_{\{i,j\} \in E(H)} d(V_i,V_j) \pm \varepsilon \binom t2 \right) \prod_{i=1}^t \ab{V_i}.
	\]
\end{lemma}

\begin{proof}[Proof of \cref{lem:regularize-cliques}]
	Let $\varepsilon = \gamma^2/(9\binom{k+2}2)$. 
	Let $\beta=\beta(\varepsilon,k+1)$ be the constant from \cref{lem:dlr}, and let $\zeta = \beta \gamma/e^{k+1}$. Note that $\zeta \geq \gamma^{k^{50}\gamma^{-10}}$, as claimed.

	We now choose a uniformly random partition of $V(G)$ into $k+1$ parts, by assigning each vertex independently and uniformly to one of the $k+1$ parts. The probability that a given $K_{k+1}$ in $G$ spans the $k+1$ parts is exactly $(k+1)!/(k+1)^{k+1}$. Linearity of expectation then implies that there exists a $(k+1)$-partite subgraph of $G$ with parts $V_1,V_2,\ldots,V_{k+1}$ which contains at least $\gamma (k+1)! (N/(k+1))^{k+1} \geq \gamma(k+1)! \prod_{i=1}^{k+1} \ab{V_i}$ copies of $K_{k+1}$. This in particular implies that $\min_i\ab{V_i}\geq \gamma(k+1)! N/(k+1)^{k+1}\geq \gamma N/e^{k+1}$, since the total number of $K_{k+1}$ in this subgraph is at most $N^k\min_i\ab{V_i}$. Let $F$ be the $(k+1)$-partite subgraph of $G$ with these parts. By \cref{lem:dlr}, there is an $\varepsilon$-regular cylinder partition $\K$ of $V_1 \times \dotsb \times V_{k+1}$, with $\ab{V_i(K)} \geq \flo{\beta \ab{V_i}}\geq \flo{\beta \gamma N/e^{k+1}} = \flo{\zeta N}$ for all $K \in \K$ and $i \in [k+1]$. By \cref{lem:counting}, if $K$ is an $\varepsilon$-regular cylinder, then the number of $K_{k+1}$ contained in $K$ is at most
	\[
		\left(\prod_{1 \leq i<j\leq k+1} d(V_i(K),V_j(K))+ \varepsilon \binom {k+1}2\right) \prod_{i=1}^{k+1} \ab{V_i(K)}.
	\]
	On the other hand, since at most an $\varepsilon$-fraction of the $(k+1)$-tuples in $V_1 \times \dotsb \times V_{k+1}$ are contained in non-$\varepsilon$-regular cylinders, we find that the number of $K_{k+1}$ contained in irregular cylinders is at most $\varepsilon \prod_i\ab{V_i}$. Adding these two facts up over all the cylinders $K \in \K$, we find that the total number of $K_{k+1}$ in $V_1 \times \dotsb \times V_{k+1}$ is at most
	\[
		\varepsilon \prod_{i=1}^{k+1}\ab{V_i}+\sum_{K\text{ $\varepsilon$-regular}}\left(\prod_{1 \leq i<j\leq k+1} d(V_i(K),V_j(K))+ \varepsilon \binom {k+1}2\right) \prod_{i=1}^{k+1} \ab{V_i(K)}.
	\]
	On the other hand, we also know that this number is at least $\gamma(k+1)!\prod_i \ab{V_i}$. Therefore, we find some $\varepsilon$-regular cylinder $K$ which satisfies
	\begin{equation}\label{eq:prod-density-lb}
		\prod_{1 \leq i<j \leq k+1} d(V_i(K),V_j(K)) \geq \gamma(k+1)!-\varepsilon - \varepsilon\binom{k+1}2 \geq  \gamma ((k+1)!-1) ,
	\end{equation}
	by our choice of $\varepsilon$.
	We can then rewrite (\ref{eq:prod-density-lb}) as
	\begin{equation}\label{eq:density-avg}
		\prod_{i=1}^{k+1} \left( \prod_{\substack{1 \leq j \leq k+1\\j \neq i}} d(V_i(K),V_j(K)) \right) ^{1/2} \geq\gamma ((k+1)!-1) .
	\end{equation}
	Suppose without loss of generality that the parenthesized term in (\ref{eq:density-avg}) is maximized for $i =k+1$. This implies that
	\begin{equation}\label{eq:X-dens-lb}
		\prod_{j=1}^{k} d(V_j(K),V_{k+1}(K)) \geq {\left( \gamma ((k+1)!-1) \right)^{2/(k+1)}} \geq \frac 52 \gamma^{2/(k+1)},
	\end{equation}
	where the final inequality uses that $((k+1)!-1)^{2/(k+1)}$ is an increasing function of $k$, and its value at $k=2$ is $5^{2/3}>5/2$. We now set $C=V_{k+1}(K)$ and $A_i=V_i(K)$ for $1 \leq i \leq k$, recalling that we proved that all these sets have size at least $\flo{\zeta N}$, as needed.
	We note that by \eqref{eq:prod-density-lb} and our choice of $\varepsilon$, we have that 
	\begin{align}
		\varepsilon \binom{k}2 &\leq \frac 1{9} \prod_{1 \leq i<j \leq k} d(A_i,A_j),\label{eq:eta-bound1}\\  
		\varepsilon \binom {k+1}2 &\leq \frac 1{9}\left( \prod_{1 \leq i<j\leq k} d(A_i,A_j) \right) \left( \prod_{i=1}^k d(A_i,C) \right), &\text{ and} \label{eq:eta-bound2}\\ 
		\varepsilon \binom{k+2}2 &\leq \frac 1{9} \left( \prod_{1 \leq i<j \leq k} d(A_i,A_j) \right) \left( \prod_{i=1}^{k} d(A_i,C)^2 \right).\label{eq:eta-bound3}
	\end{align}
	Indeed, the final inequality follows immediately from \eqref{eq:prod-density-lb}, \eqref{eq:X-dens-lb}, and our choice of $\varepsilon$, and it immediately implies the other two.

	Let $N_{k}$ denote the number of labeled $K_{k}$ with one vertex in each of $A_1,\ldots,A_{k}$, and let $N_{k+1}$ denote the number of labeled $K_{k+1}$ with one vertex in each of $A_1,\ldots,A_k,C$. Additionally, let $K_{k+2}^-$ denote $K_{k+2}$ with an edge deleted, and let $N_{k+2}^-$ denote the number of labeled homomorphic copies of $K_{k+2}^-$ such that the two non-adjacent vertices are in $C$, and the $k$ remaining vertices are in $A_1,\ldots,A_{k}$, respectively.

	Let $\Q$ be a uniformly randomly chosen $K_{k}$ among $A_1,\ldots,A_{k}$, and let $\Z$ denote the number of extensions of $\Q$ in $C$, i.e.\ the number of vertices $c \in C$ such that $\Q \cup \{c\}$ is a $K_{k+1}$. We observe that
	\[
		\E[\Z] = \frac{1}{N_{k}}\sum_{Q}\ab{\{c \in C:Q \cup \{c\}\text{ is a }K_{k+1}\}} = \frac{N_{k+1}}{N_{k}},
	\]
	where the sum is over all choices of $\Q$, that is all cliques $Q$ among $A_1,\ldots,A_{k}$. Similarly, we have that
	\[
		\E[\Z^2] = \frac{1}{N_{k}} \sum_Q \ab{\{c_1,c_2 \in C:Q \cup \{c_1\}, Q \cup \{c_2\}\text{ are both }K_{k+1}\}} = \frac{N_{k+2}^-}{N_{k}}.
	\]
	By Lemma \ref{lem:counting} and the bounds \eqref{eq:eta-bound1}, \eqref{eq:eta-bound2}, and \eqref{eq:eta-bound3}, we have that
	\begin{align*}
		\frac 89 \prod_{1 \leq i<j \leq k} d(A_i,A_j) \prod_{i=1}^{k} \ab{A_i}  \leq N_{k} &\leq \frac {10}{9} \prod_{1 \leq i<j \leq k} d(A_i,A_j) \prod_{i=1}^{k} \ab{A_i}, \\
		 N_{k+1}&\geq \frac 89 \left( \prod_{1 \leq i<j \leq k} d(A_i,A_j) \right)\left( \prod_{i=1}^{k} d(A_i,C) \right) \left( \prod_{i=1}^k \ab{A_i} \right)\ab C\\
		N_{k+2}^- &\leq \frac {10}{9} \left( \prod_{1 \leq i<j \leq k} d(A_i,A_j) \right) \left( \prod_{i=1}^{k} d(A_i,C)^2 \right) \left( \prod_{i=1}^{k} \ab{A_i} \right) \ab{C}^2.
	\end{align*}
	Therefore, we deduce that
	\begin{align*}
		\E[\Z]=\frac{N_{k+1}}{N_{k}} \geq \frac 45 \left( \prod_{i=1}^{k} d(A_i,C) \right) \ab{C} \qquad \text{ and } \qquad
		\frac{\E[\Z]^2}{\E[\Z^2]}&= \frac{N_{k+1}^2}{N_{k} N_{k+2}^-} \geq \frac{16}{25}.
	\end{align*}
	Let
	\[
		\tau = \frac 25 \prod_{i=1}^{k} d(A_i,C) 
	\]
	so that $\tau \geq \gamma^{2/(k+1)}$ by \eqref{eq:X-dens-lb}, and $\tau \leq \frac 25 \leq \frac 12$ by definition. Let $\G$ be the
	$k$-mixed graph on $(\A,C)$ whose edges between $C$ and $A_1\cup \dots \cup A_k$ are all edges of $G$ between these sets, and whose edges in $\A$ consist of all $K_k$ which have at least $\tau\ab C$ extensions in $C$. By definition, we have $\tau_C(\G)\geq \tau \geq \gamma^{2/(k+1)}$.

	To compute $e_\G(\A)$, we note that $e_\G(\A)=N_k \pr(\Z \geq \tau\ab C)$. 	
	By the Paley--Zygmund inequality, we have
	\[
		\pr(\Z \geq \tau \ab{C}) \geq \pr \left( \Z \geq \frac 12 \E[\Z] \right) \geq \frac 14 \frac{\E[\Z]^2}{\E[\Z^2]} \geq \frac{4}{25},
	\]
	and therefore
	\[
		d_\G(\A)= \frac{N_k \pr(\Z \geq \tau \ab C)}{\prod_{i=1}^k\ab{A_i}} \geq \frac{4}{25}\cdot \frac 89 \prod_{1 \leq i<j \leq k} d(A_i,A_j) .
	\]
	This implies that
	\[
		d_\G(\A) \tau \geq \frac{4}{25}\cdot \frac 89 \cdot \frac 25 \left( \prod_{1 \leq i<j \leq k} d(A_i,A_j) \right) \left( \prod_{i=1}^{k} d(A_i,C)  \right) \geq \frac{1}{20} \cdot \gamma ((k+1)!-1) \geq \frac \gamma 4,
	\]
	where the second inequality uses \eqref{eq:prod-density-lb} and the final inequality uses that $k \geq 2$.
\end{proof}

\end{document}